\documentclass[11pt,letterpaper]{article}
\usepackage[utf8]{inputenc}
\usepackage[english]{babel}

\usepackage{amsmath}
\usepackage{amsfonts}
\usepackage{amssymb}
\usepackage{gensymb}
\usepackage{rotating}
\usepackage[titletoc]{appendix}
\usepackage{array}
\usepackage{float}
\usepackage{graphicx}
\usepackage{caption}
\usepackage{subcaption}
\usepackage{cleveref}
\crefname{section}{§}{§§}
\Crefname{section}{§}{§§}
\usepackage{xcolor}
\usepackage{authblk}

\usepackage{xr}
\usepackage{geometry}
\usepackage{times}
\usepackage{lineno}

\usepackage{amsthm}
\newtheorem{theorem}{Theorem}[section]

\newtheorem{proposition}[theorem]{Proposition}

\newtheorem{remark}{Remark}[section]

\definecolor{mve}{rgb}{0.7,0.35,0.15}
\definecolor{brght}{rgb}{0.825,0.2625,0.15}
\definecolor{yello}{rgb}{1,0.925,0.65}
\definecolor{bluu}{rgb}{0.65, 0.95, 1}
\definecolor{bluu2}{rgb}{0.2, 0.5, 0.8}

\title{Local sensitivity analysis of the `Membrane shape equation' derived from the Helfrich energy}

\author[1]{P. Rangamani\thanks{prangamani@ucsd.edu}} 
\affil[1]{Department of Mechanical and Aerospace Engineering, University of California San Diego}
\author[2]{A. Behzadan} 
\author[2]{M. Holst} 
\affil[2]{Department of Mathematics, University of California San Diego}

\begin{document}
\maketitle

\begin{abstract}%
{The Helfrich energy is commonly used to model the elastic bending energy of lipid bilayers in membrane mechanics. 
The governing differential equations for certain geometric characteristics of the shape of the membrane can be obtained by applying variational methods (minimization principles) to the Helfrich energy functional and are well-studied in the axisymmetric framework.
However, the Helfrich energy functional and the resulting differential equations involve a number of parameters, and there is little explanation of the choice of parameters in the literature, particularly with respect to the choice of the  ``spontaneous curvature" term that appears in the functional.
In this paper, we present a careful analytical and numerical study of certain aspects of parametric sensitivity of Helfrich's model.
Using simulations of specific model systems, we demonstrate the application of our scheme to the formation of spherical buds and pearled shapes in membrane vesicles.
}
\end{abstract}

\textbf{Keywords}
Helfrich model; membrane curvature; spontaneous curvature; parametric sensitivity analysis.

\tableofcontents



\newpage


\section{Introduction}\label{sec:Introduction}

The elastic behavior of lipid bilayers has been studied using mechanical models for fifty years or more \cite{Canham1970,Helfrich1973}. 
A pseudoelastic strain energy functional was first described by Canham in 1970 \cite{Canham1970} as a way to explain the biconcave shape of red blood cells and subsequently by Helfrich \cite{Helfrich1973} to describe the mechanical behavior of lipid bilayers in various situations. 
This energy functional has now become the most accepted model for describing the mechanical properties of the cell membrane. 
This model has been used to study the shapes associated with whole cells, particularly that of the red blood cell \cite{Canham1970,alimohamadi2020,turlier2016}.
Subsequently, the Helfrich energy has  been used to study the different shapes associated with vesicles as a function of pressure, volume , and membrane composition \cite{Seifert1997}. 

While the original model proposed by Helfrich considered the lipid bilayer as a thin shell with negligible thickness, over the years, there have been many mathematical developments to represent the different physical properties of biological membranes.
Some of these include the area difference model \cite{miao1994budding} and the mattress model \cite{mouritsen1984mattress}.
Among these, the spontaneous curvature model has been one of the most popularly used models to represent the asymmetry between the two leaflets of the bilayer \Cref{fig:Figure1}A \cite{alimohamadi2018}.
The idea of spontaneous curvature to represent the asymmetry between the two leaflets of the lipid bilayer was first introduced by Helfrich in his seminal 1973 paper \cite{Helfrich1973}.
Subsequently, this term has since been used to capture compositional asymmetry, protein-induced spontaneous curvature,  coat proteins etc. (see  \cite{alimohamadi2018} for a more detailed discussion).


For many problems of biophysical interest, it is necessary to consider the heterogeneity in composition across cell membranes \cite{kozlov2014mechanisms}. 
In such cases, the spontaneous curvature function is often modeled as a spatially varying quantity rather than as a constant to represent different membrane domains without introducing a discontinuity for computational purposes\cite{hassinger2017,agrawal2009modeling}.
Alongside theoretical modeling efforts, there have been significant advances in computational methods for solving the partial differential equations resulting from the minimization procedure of the Helfrich energy \cite{vasan2020,sauer2014}. 
A vast majority of the simulations are implemented under the assumption of axisymmetry; this assumption enables us to transform the partial differential equations describing the shape of the membrane into a system of ordinary differential equations, which are then equipped with appropriate boundary conditions to be solved \cite{luke1982} (\Cref{fig:Figure1}B). 
However, a major challenge associated with such simulations remains the number of free parameters associated with the spontaneous curvature function. 

We and others have found that the specific choice of the spontaneous curvature function and the resulting parameters play an important role in determining the shape of the membrane \cite{hassinger2017,walani2015endocytic}. 
For certain parameters such as membrane bending modulus, there exist sufficient experimental measurements to establish a range of physically relevant values \cite{waugh1982surface}. 
However, for the spontaneous curvature function, such measurements are limited or don't exist in forms that are always amenable to modeling.
As a result, throughout the literature, spontaneous curvature has been represented by different functions, using a wide range of parameters values.
 Thus, it seems there is a need for a better understanding of the sensitivity of the spontaneous curvature model with respect to the various parameters involved.

To address these issues and gain some insight into the role of parameters, as a first step, in this work we study the local sensitivity of solutions of certain equations associated with the Helfrich energy model with spontaneous curvature.
We note that while parametric sensitivity analysis methods are well-documented for initial value problems (IVPs) \cite{varma1999parametric}, in this work the problems we will consider are boundary value problems (BVPs). 
We will show that parametric sensitivity analysis can provide valuable insight into the impact of different parameters on energy minimization and a tool for better understanding critical phenomena associated with the shape of the membrane.

In what follows, we present a summary of the Helfrich model with spontaneous curvature and axisymmetric parametrization in \S \ref{sec:Helfrich_energy} and \S \ref{sec:axsi-param}, development of the parametric sensitivity analysis method in \S \ref{sec:over-sensitivity-analysis}, present some numerical results in \S \ref{sec:numer-results}, and end with our interpretations and conclusions in \S \ref{sec:conclusion}.

\begin{figure}[!!h]
    \centering
    \includegraphics[width=\textwidth]{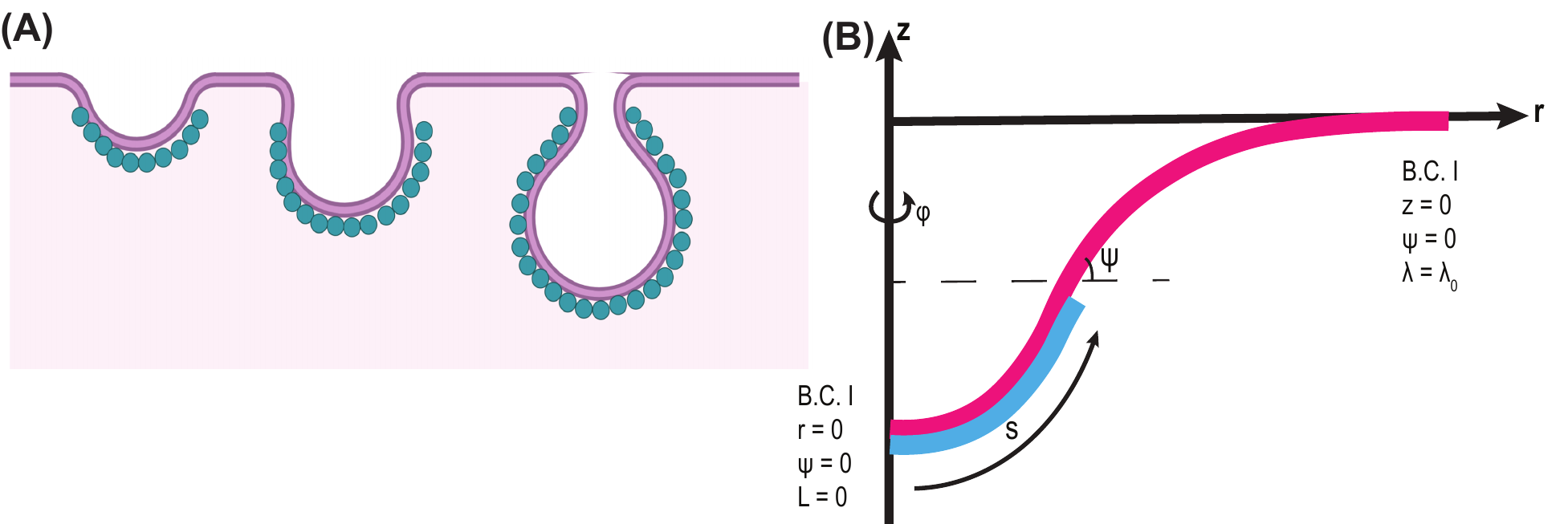}
    \caption{(A) Schematic depicting a sequence of membrane shapes associated with bending induced by a protein coat represented by blue circles. This protein coat or other asymmetries in the leaflet are often modeled using a spontaneous curvature. (B) Axisymmetric coordinates used to simulate the governing equations and the corresponding boundary conditions.}
    \label{fig:Figure1}
\end{figure}

\section{Overview of the Helfrich energy with spontaneous curvature}\label{sec:Helfrich_energy}
The Helfrich energy serves as the constitutive equation for the lipid bilayer. We use a modified version of the Helfrich energy that includes spatially-varying spontaneous curvature $C$ as opposed to a constant uniform value, as in \cite{steigmann1999fluid,agrawal2009boundary,hassinger2017}.

\begin{align}
w = \kappa \left[ H - C\right]^2 + \kappa_G K,
\label{eq:Helfrich}
\end{align}

\noindent where $w$ is the energy per unit area, $\kappa$ is the bending modulus, $H$ is the mean curvature, $\kappa_G$ is the Gaussian modulus, and $K$ is the Gaussian curvature.  
 This form of the energy density accommodates the local heterogeneity in the spontaneous curvature $C$. Note that $w$ differs from the standard Helfrich energy \cite{Helfrich1973} by a factor of $2$, which is accounted for by using the value of $\kappa$ to be twice that of the standard bending modulus typically  encountered in the literature.

\newpage



 



\subsection{Equations of motion}\label{sec:s-eqnmotion}

We refer the interested reader to \cite{steigmann1999fluid} for a detailed derivation of the governing equations.
We make the following simplifying assumptions for simplicity in our model. We assume that the bending modulus and Gaussian modulus are uniform, the pressure difference across the membrane is zero, and there are no externally applied forces.
Furthermore, we assume that the membrane is areally incompressible and introduce a Lagrange multiplier $\lambda$ to impose this constraint. 
It can be shown that among all areally incompressible surfaces, the one that minimizes the Helfrich energy functional has mean and Gaussian curvatures $H$ and $K$ satisfying:

\begin{align} 
\kappa \Delta \left[\left(H - C \right)\right] + 2 \kappa \left( H - C \right) \left(2 H^2 - K \right) - 2 \kappa H \left( H - C \right)^2 =  2 \lambda H  \label{eq:shape_simplified}
\end{align} 
and 
\begin{align} 
\frac{\partial\lambda}{\partial \theta^\alpha} =  2 \kappa \left( H - C \right)\frac{\partial C}{\partial \theta^\alpha}\,,
\label{eq:lambda_simplified}
\end{align}
where $\theta^{\alpha}$ ($\alpha=1, 2$) denotes the surface coordinates and $\Delta$ is the Laplace-Beltrami operator. \Cref{eq:shape_simplified} gives the normal force balance on the membrane. 
\Cref{eq:lambda_simplified} represents the tangential force balance along the membrane. 
For a heterogeneous membrane (non-uniform or coordinate-dependent $C(\theta^\alpha)$), this equation also represents the variation of $\lambda$ along the surface \cite{steigmann1999fluid}. 
Detailed discussions of interpretation of $\lambda$ are given in \cite{steigmann1999fluid, rangamani2013interaction, rangamani2014protein}.

\subsection{Choice of spontaneous curvature function}\label{sec:s-Spontaneous-choice}
In the existing literature, for cases where the situation under consideration can be classified as an axisymmetric problem, the spontaneous curvature function, $C(\theta^\alpha)$, is often represented using a hyperbolic tangent function to capture differences in the `coated' versus `uncoated' regions. 
In this work, we use two forms of this function. 
Each form involves certain free parameters and we perform our parametric sensitivity analysis to understand the extent to which these parameters can affect the elastic bending energy and the shape of the membrane.
These functions are given as

\begin{align}
& \textrm{Type I:}\qquad  C(u)=0.5C_0\big[1-\tanh{[\xi(u-u_0)]}\big],\\
 & \textrm{Type II:}\qquad C(u)=-0.5C_0\frac{u-u_0}{u_0}\big[1-\tanh{[\xi(u-u_0)]}\big] . 
 \end{align}
 
In the above formulae, we used the generic variable $u$ instead of $\theta^\alpha$.
This variable can be viewed as one of the coordinates used in the parametrization of the surface of revolution representing the membrane as will be described in the next section.

\section{Axisymmetric parametrization}\label{sec:axsi-param}

Under the assumption of axisymmetry, which in part implies that the surface representing the membrane is a surface of revolution obtained by rotating a regular curve about an axis, the governing equations shown in \Cref{eq:shape_simplified,eq:lambda_simplified} can be recast as a system of ordinary differential equations. 
When equipped with appropriate boundary conditions, reflecting the geometric and physical constraints of the problem, the solution of the resulting boundary value problem will determine the shape of the membrane. 
In what follows, we first describe this viewpoint in more detail, and then we summarize the various ways our ODE system can be represented depending on the choice of arc-length or area parametrization. \\

Under the assumption of axisymmetry, without loss of generality, we may assume that the surface of the membrane is generated by rotating a curve in the right half of the $rz$-plane, about the $z$-axis. 
We let $\textbf{p}:[0,S]\rightarrow \mathbb{R}^2$ defined by $s\mapsto (r(s),z(s))$
denote the arc-length parametrization of the generating curve (or the profile curve). 
Then 
 \begin{equation}
     \textbf{r}(s,\varphi)=\big(r(s)\cos{\varphi}, r(s)\sin{\varphi}, z(s)\big)
 \end{equation}
is a parametrization of the surface of membrane (here $0\leq s\leq S$ and $0\leq \varphi\leq 2\pi$). 
Since $[r'(s)]^2+[z'(s)]^2=1$, there exists a function $\Psi(s)$ such that
 \begin{equation}
     r'(s)=\cos{\Psi{(s)}},\qquad z'(s)=\sin{\Psi{(s)}}.
 \end{equation}
$\Psi$ can be viewed as the angle made by the curve with the horizontal.

Using well-known formulas for the Gaussian and mean curvatures of a surface of revolution (see, for example, \cite{gray2006geometry}), it is easy to show that if we orient the surface with the unit normal vector that points in the direction of  $-\frac{\partial \textbf{r}}{\partial \textbf{s}}\times \frac{\partial \textbf{r}}{\partial \varphi}$, then
\begin{align}
    &r\Psi'=2rH-\sin{\Psi},\\
    &K=H^2-(H-\frac{\sin{\Psi}}{r})^2.
\end{align}
In order to be able to write the final equations as first-order equations, we introduce the auxiliary variable $L$ as follows:
\begin{equation}
     L=r\big[H-C\big]'.
\end{equation}
Now, using the above equalities and  \Cref{eq:shape_simplified},  it is easy to see that 

\begin{align}
    & H'(s)=r(s)^{-1}L(s)+C'(s),\\
    &r(s)^{-1}L'(s)=2H(s)\big[(H(s)-C(s))^2+\frac{\lambda(s)}{\kappa}\big] \\
    \nonumber
    &\qquad \qquad\qquad\qquad  -2(H(s)-C(s))
    \big[H(s)^2+\big(H(s)-\frac{\sin{\Psi(s)}}{r(s)}\big)^2\big]. 
\end{align}
In the above, we used the fact that if $f$ is any function defined on our surface of revolution whose values depend only on $s$, then  
\begin{equation*}
\Delta f= \frac{[r(s) f'(s)]'}{r(s)}.
\end{equation*}
Finally, \Cref{eq:lambda_simplified} gives
\begin{align}
    \lambda'(s)=2\kappa (H(s)-C(s))C'(s).
\end{align}

\begin{remark}
In what has been discussed so far, we have referred to the \emph{assumption of axisymmetry} a number of times without clearly explaining what this assumption entails. Now we are at a position to give a careful description of this key assumption. In the present work, we say that our problem falls into the axisymmetric category if and only if 
\begin{enumerate}
    \item the surface of the membrane is a surface of revolution parametrized by $\textbf{r}(s,\varphi)$ as described above, and
    \item the values of the spontaneous curvature function $C$ depend only on $s$.
\end{enumerate}
\end{remark}

\subsection{Arc-length formulation}\label{sec:s-arclength-1}

In this formulation, as described above, $s$ is the arc-length along the membrane and the unknown functions are 
\begin{align}
r(s),\,z(s),\,\Psi(s),\, H(s), \, L(s),\,\lambda(s).
\end{align}
These functions must satisfy the following system of ODEs on the interval $[0,S]$:
\begin{subequations}
\begin{align}
    & r'=\cos{\Psi},\quad z'=\sin{\Psi},\quad r\Psi'=2rH-\sin{\Psi},\quad H'=r^{-1}L+C',\quad\\ &r^{-1}L'=2H\big[(H-C)^2+\frac{\lambda}{\kappa}\big]-2(H-C)\big[H^2+(H-r^{-1}\sin{\Psi})^2\big],\\
    & \lambda'=2\kappa(H-C)C'.
\end{align}
\end{subequations}

Before we attempt to numerically solve the above system of equations, we need to provide the system with appropriate boundary conditions. 
Certain suitable classes of boundary conditions for the above system  will be introduced in \S \ref{sec:numer-results}.

\subsection{Dimensionless arc-length formulation}\label{sec:s-arclength-2}
In this formulation, we fix two positive constants $R_0$ and $\kappa_0$ that will be used to nondimensionalize the variables in the arc-length formulation by setting $t=\frac{s}{R_0}$,  $\tilde{\kappa}=\frac{\kappa}{\kappa_0}$, and defining:
\begin{subequations}
\begin{align}
&x(t)=\frac{1}{R_0}r(R_0t),\quad y(t)=\frac{1}{R_0}z(R_0t),\quad \psi(t)=\Psi(R_0t),\quad h(t)=R_0H(R_0t),\\
&c(t)=R_0C(R_0t),\quad l(t)=R_0L(R_0t),\quad \tilde{\lambda}(t)=\frac{R_0^2}{\kappa_0}\lambda(R_0t).
\end{align}
\end{subequations}
Note that $s\in [0,S]$ and $t\in [0,T]$ where $S=R_0T$.\\
A simple application of chain rule shows that the six dimensionless unknown functions $x(t)$,
$ y(t)$, $\psi(t)$, $h(t)$, $l(t)$, and $\tilde{\lambda}(t)$ must satisfy the following system of ODEs on the interval $[0,T]$:
\begin{subequations}
\begin{align}
& \dot{x}=\cos{\psi},\quad \dot{y}=\sin{\psi},\quad x\dot{\psi}=2xh-\sin{\psi},\quad \dot{h}=x^{-1}l+\dot{c},\\
& x^{-1}\dot{l}=2h\big[(h-c)^2+\frac{\tilde{\lambda}}{\tilde{\kappa}}\big]-2(h-c)\big[h^2+(h-x^{-1}\sin{\psi})^2\big],\\
& \dot{\tilde{\lambda}}=2\tilde{\kappa}(h-c)\dot{c}.
\end{align}
(Here dot denotes the derivative with respect to $t$.)
\end{subequations}
\begin{remark}
It is worth mentioning that if we assume $C(s)$ has the following form:
\begin{equation}
    C(s)=0.5C_0\big[1-\tanh{[\xi (s-s_0)]}\big],
\end{equation}
where $C_0$, $\xi$, and $s_0$ are certain constants, then
\begin{align}
c(t)=R_0C(R_0t)=0.5C_0R_0\big[1-\tanh{[\xi (R_0t-s_0)]}\big]=
0.5C_0R_0\big[1-\tanh{[\xi R_0 (t-\frac{s_0}{R_0})]}\big].
\end{align}
So if we let $\gamma=R_0 \xi$ and $t_0=\frac{s_0}{R_0}$, then
\begin{equation}
    c(t)=0.5C_0R_0\big[1-\tanh{[\gamma (t-t_0)]}\big].
\end{equation}
Similar calculations show that if we assume $C(s)$ has the form
\begin{equation}
    C(s)=(-0.5C_0)(\frac{s-s_0}{s_0})\big[1-\tanh{[\xi (s-s_0)]}\big],
\end{equation}
then
\begin{equation}
    c(t)=(-0.5C_0)R_0(\frac{t-t_0}{t_0})\big[1-\tanh{[\gamma (t-t_0)]}\big].
\end{equation}
where $\gamma=R_0 \xi$ and $t_0=\frac{s_0}{R_0}$.
\end{remark}
\begin{remark} \label{remscaling1}
In the future sections of this paper, we will be interested in the derivatives of the solution functions with respect to input parameters such as $C_0$. 
These derivatives will be referred to as sensitivities.
An application of the chain rule shows that sensitivities calculated using dimensionless variables are constant multiples of sensitivities computed using the original variables. 
For example,
\begin{align}
\frac{\partial h}{\partial C_0}\big|_{t=\hat{t}}&=R_0\frac{\partial H}{\partial C_0}\big|_{s=R_0\hat{t}},\\
    \frac{\partial h}{\partial \gamma}\big|_{t=\hat{t}}&=\frac{1}{R_0}\frac{\partial h}{\partial \xi}\big|_{t=\hat{t}}=\frac{1}{R_0}R_0\frac{\partial H}{\partial \xi}\big|_{s=R_0\hat{t}}=\frac{\partial H}{\partial \xi}\big|_{s=R_0\hat{t}},\\
    \frac{\partial h}{\partial t_0}\big|_{t=\hat{t}}&=\frac{1}{1/R_0}\frac{\partial h}{\partial s_0}\big|_{t=\hat{t}}=\frac{1}{1/R_0}R_0\frac{\partial H}{\partial s_0}\big|_{s=R_0\hat{t}}=R_0^2\frac{\partial H}{\partial s_0}\big|_{s=R_0\hat{t}}.
    \end{align}
\end{remark}

\subsection{Area formulation}\label{sec:s-area-1}
Here we introduce the new variable $a$ as the area of the surface of revolution produced by rotating the segment of the curve spanned as the arc-length varies from $0$ to $s$, that is,
\begin{equation}
    a(s)=\int_0^s2\pi r(u)\,du.
\end{equation}
Since there is a one-to-one relationship between $a$ and $s$, we can view the six unknown functions as functions of $a$ rather than $s$.
As it pointed out in \cite{hassinger2017}, this formulation has the advantage of prescribing the total area of the membrane as the domain size (rather than the corresponding arc-length) which is more physical and amenable to laboratory measurements.

A simple application of the chain rule shows that, if we denote the total area (of the membrane) by $A$, then the unknown functions
\begin{align}
r(a),\,z(a),\,\Psi(a),\, H(a), \, L(a),\,\lambda(a)
\end{align}
must satisfy the following system of ODEs on the interval $[0,A]$:
\begin{subequations}
\begin{align}
    & 2\pi r r'=\cos{\Psi},\quad 2\pi r z'=\sin{\Psi},\quad 2\pi r^2\Psi'=2rH-\sin{\Psi},\quad 2\pi r^2 H'=L+2\pi r^2C',\quad\\ &2\pi L'=2H\big[(H-C)^2+\frac{\lambda}{\kappa}\big]-2(H-C)\big[H^2+(H-r^{-1}\sin{\Psi})^2\big],\\
    & \lambda'=2 \kappa (H-C) C'.
\end{align}
\end{subequations}
\noindent As noted above, we will discuss suitable boundary conditions in \S \ref{sec:numer-results}.\begin{remark}\label{remmay242}
We emphasize that in the above formulae, the functions $r$, $z$, $\Psi$, $H$, $L$, $\lambda$, and $C$ which originally were introduced as functions of $s$, are viewed as functions of the variable $a$; indeed, if we let $g(s)=\int_0^s 2\pi r(u)\,du$, then any function of $s$ can be written as a function of $a$ by replacing each instance of $s$ in the expression of the function by $g^{-1}(a)$; so if we wanted to be completely rigorous, in the above, instead of $r$, $z$, $\Psi$, $H$, $L$, $\lambda$, and $C$, we should have written $\hat{r}$, $\hat{z}$, $\hat{\Psi}$, $\hat{H}$, $\hat{L}$, $\hat{\lambda}$, and $\hat{C}$, where $\hat{r}=r(g^{-1}(a))$, $\hat{z}=z(g^{-1}(a))$, etc. It is customary to abuse notation and denote these new variables by the same symbols as the original variables. In \S \ref{sec:s-insights} we will make use of this very elementary observation to better understand the relationship between the simulation results obtained using the arc-length formulation and those obtained using the area formulation. 
\end{remark}

\subsection{Dimensionless area formulation}\label{sec:s-area-2}
In this formulation, we fix two positive constants $R_0$ and $\kappa_0$ which will be used to nondimensionalize the variables in the previous formulation by setting $\alpha=\frac{a}{2\pi R_0^2}$, $\tilde{\kappa}=\frac{\kappa}{\kappa_0}$, and defining:
\begin{subequations}
\begin{align}
&x(\alpha)=\frac{1}{R_0}r(2\pi R_0^2 \alpha),\quad y(\alpha)=\frac{1}{R_0}z(2\pi R_0^2 \alpha),\quad \psi(\alpha)=\Psi(2\pi R_0^2 \alpha),\quad h(\alpha)=R_0H(2\pi R_0^2 \alpha),\\
&c(\alpha)=R_0C(2\pi R_0^2 \alpha),\quad l(\alpha)=R_0L(2\pi R_0^2 \alpha),\quad \tilde{\lambda}(\alpha)=\frac{R_0^2}{\kappa_0}\lambda(2\pi R_0^2 \alpha).
\end{align}
\end{subequations}
Note that $a\in [0,A]$ and $\alpha\in [0,\alpha_{max}]$ where $A=2\pi R_0^2 \alpha_{max}$.\\
A simple application of chain rule shows that the six dimensionless unknown functions $x(\alpha)$, $y(\alpha)$, $\psi(\alpha)$ , $h(\alpha)$, $l(\alpha)$, and  $\tilde{\lambda}(\alpha)$ must satisfy the following system of ODEs on the interval $[0,\alpha_{max}]$:
\begin{subequations}
\begin{align}
& x\dot{x}=\cos{\psi},\quad x\dot{y}=\sin{\psi},\quad x^2\dot{\psi}=2xh-\sin{\psi},\quad x^2\dot{h}=l+x^2 \dot{c},\\
& \dot{l}=2h\big[(h-c)^2+\frac{\tilde{\lambda}}{\tilde{\kappa}}\big]-2(h-c)\big[h^2+(h-x^{-1}\sin{\psi})^2\big],\\
& \dot{\tilde{\lambda}}=2\tilde{\kappa}(h-c)\dot{c}.
\end{align}
\end{subequations}
\begin{remark}
It is worth mentioning that if we assume $C(a)$ has the following form:
\begin{equation}
    C(a)=(-0.5C_0)(\frac{a-a_0}{a_0})\big[1-\tanh{[\xi (a-a_0)]}\big],
\end{equation}
where $C_0$, $\xi$, and $a_0$ are certain constants, then
\begin{align}
c(\alpha)&=R_0C(2\pi R_0^2 \alpha)=-0.5C_0R_0
(\frac{2\pi R_0^2 \alpha-a_0}{a_0})\big[1-\tanh{[\xi (2\pi R_0^2 \alpha-a_0)]}\big]\notag\\
&=
-0.5C_0R_0\frac{\alpha-a_0/(2\pi R_0^2 \alpha)}{a_0/(2\pi R_0^2 \alpha)}
\big[1-\tanh{[\xi 2\pi R_0^2 (\alpha-\frac{a_0}{2\pi R_0^2})]}\big].
\end{align}
So if we let $\gamma=2\pi R_0^2 \xi$ and $\alpha_0=\frac{a_0}{2\pi R_0^2}$, then
\begin{equation}
    c(\alpha)=-0.5C_0R_0(\frac{\alpha-\alpha_0}{\alpha_0})\big[1-\tanh{[\gamma (\alpha-\alpha_0)]}\big].
\end{equation}
Similarly, one can show that if $C(a)$ has the form
\begin{equation}
    C(a)=(0.5C_0)\big[1-\tanh{[\xi (a-a_0)]}\big],
\end{equation}
then
\begin{equation}
    c(\alpha)=0.5C_0R_0\big[1-\tanh{[\gamma (\alpha-\alpha_0)]}\big],
\end{equation}
where  $\gamma=2\pi R_0^2 \xi$ and $\alpha_0=\frac{a_0}{2\pi R_0^2}$.
\end{remark}

\begin{remark} \label{remscaling2}
\
The same argument as the one discussed in Remark \ref{remscaling1} shows that sensitivities calculated using dimensionless variables are constant multiples of sensitivities computed using the original variables.
\end{remark}

We conclude this section by making a few comments about the choice of the positive constants $R_0$ and $\kappa_0$ that are used in order to nondimensionalize the variables. 
The following observations indeed show existence of certain symmetries in the system of ODEs under consideration. Generally, it is true that symmetries of differential equations can be used to gain valuable information about the equations under consideration.
In \S \ref{sec:disprove}, we will use these observations to examine a conjecture about the solutions of our system of equations.

\begin{enumerate}
    \item Let $\eta$ be any positive number. It is easy to see that $(x,y,\psi, h, l, \tilde{\lambda})$ is a solution of the ODE system using $(R_0, \kappa_0)$ if and only if $(x,y,\psi, h, l, \eta\tilde{\lambda})$ is a solution of the system using $(R_0, \eta\kappa_0)$. 
    That is, changing $\kappa_0$ merely amounts to scaling $\tilde{\lambda}$. 
    In our numerical experiments we choose $\kappa_0=\kappa$ so that $\tilde{\kappa}=1$.
    \item For the dimensionless arc-length formulation, $(x(t),y(t),\psi(t), h(t), l(t), \tilde{\lambda}(t))$ is a solution of the ODE system with $c(t)$ defined using parameters $C_0,R_0, \gamma, t_0$ on the interval $[0,T]$ using $(R_0, \kappa_0)$ if and only if $(\frac{1}{\eta}x(\eta t),\frac{1}{\eta}y(\eta t),\psi(\eta t), \eta h(\eta t), \eta l(\eta t), \eta^2 \tilde{\lambda}(\eta t))$ is a solution of the ODE system with $c(t)$ defined using parameters $C_0,\eta R_0,\eta \gamma, \frac{t_0}{\eta}$ on the interval $[0,\frac{T}{\eta}]$ using $(\eta R_0, \kappa_0)$.
    \item For the dimensionless area formulation, $(x(\alpha),y(\alpha),\psi(\alpha), h(\alpha), l(\alpha), \tilde{\lambda}(\alpha))$ is a solution of the ODE system with $c(\alpha)$ defined using parameters $C_0,R_0, \gamma, \alpha_0$ on the interval $[0,\alpha_{max}]$ using $(R_0, \kappa_0)$ if and only if $(\frac{1}{\eta}x( \eta^2 \alpha),\frac{1}{\eta} y( \eta^2 \alpha),\psi( \eta^2 \alpha), \eta h(\eta^2 \alpha), \eta l(\eta^2 \alpha), \eta^2 \tilde{\lambda}(\eta^2  \alpha))$ is a solution of the ODE system with $c(\alpha)$ defined using parameters $C_0, \eta R_0,\eta^2 \gamma, \frac{\alpha_0}{\eta^2}$ on the interval $[0,\frac{\alpha_{max}}{\eta^2}]$ using $(\eta R_0, \kappa_0)$.
\end{enumerate}
Finally, we remark that when solving the boundary value problem using different values of $R_0$ (and/or $\kappa_0$), in order to be able to compare the results in the sense described above, we also need to change the boundary conditions accordingly.

\section{Mathematical framework for sensitivity analysis}\label{sec:over-sensitivity-analysis}
In this section we give a brief overview of the theoretical framework of sensitivity analysis for the system of ODEs dependent on parameters. The subject is well studied in the context of initial value problems \cite{varma1999parametric}  and, as illustrated in this and in the following sections, the same key ideas can be employed to analyze the sensitivity of boundary value problems such as those that will be considered in this paper.

Consider the following system of ODEs for the unknown 
$\vec{x}(u)=[x_1(u),\ldots,x_k(u)]^T$
involving parameters 
$\vec{p}=[p_1,\ldots,p_m]^T$:
\begin{equation}\label{april261}
\textbf{F}(\vec{x}, \dot{\vec{x}},\vec{p})=\begin{bmatrix}f_1(\vec{x}, \dot{\vec{x}},\vec{p})\\\vdots\\f_k(\vec{x}, \dot{\vec{x}},\vec{p})\end{bmatrix}=\vec{0}\,.
\end{equation}
with appropriate initial or boundary conditions. 
Two key objectives of sensitivity analysis would be to provide answers to the following questions:
\begin{enumerate}
\item How can we compute the rate of change of solution with respect to each of the parameters? That is, we are interested in computing
\begin{equation*}
\frac{\partial x_i}{\partial p_j}\qquad \forall\, 1\leq i\leq k\quad \forall\, 1\leq j\leq m.
\end{equation*}
\item How can we compute the rate of change of some functional $\displaystyle W(\vec{x},\vec{p})=\int_0^{u_{max}}w(\vec{x},\vec{p})\,du$ of the solution with respect to each parameter? That is, we are interested in computing
    \begin{equation*}
    \frac{\partial W}{\partial p_j}\qquad  \forall\, 1\leq j\leq m.
    \end{equation*}
\end{enumerate}
Throughout this document we will refer to $\frac{\partial x_i}{\partial p_j}$ and $\frac{\partial W}{\partial p_j}$ as sensitivities of solution components and sensitivities of the functional, respectively. Notice that under appropriate smoothness assumptions on the functions involved, it follows from the chain rule that
\begin{equation}\label{april262}
\frac{\partial W}{\partial p_j}=\int_0^{u_{max}} \frac{\partial w}{\partial p_j}+(D_{\vec{x}}\,w)\frac{\partial \vec{x}}{\partial p_j}\,du,
\end{equation}
and so the answer to the first question can be used to answer the second question. However, as we shall see, if the ultimate objective is to just compute the sensitivity of a functional, there might be more efficient tools available for the job.\\\\
We begin with describing a standard method to answer the first question. Define the sensitivity vectors $\vec{s}_1,\cdots,\vec{s}_m$ as follows:
\begin{align}\label{april263}
\vec{s}_1=\frac{\partial \vec{x}}{\partial p_1}=\begin{bmatrix}\frac{\partial x_1}{\partial p_1}\\\vdots\\\frac{\partial x_k}{\partial p_1}\end{bmatrix},\cdots,\vec{s}_m=\frac{\partial \vec{x}}{\partial p_m}=\begin{bmatrix}\frac{\partial x_1}{\partial p_m}\\\vdots\\\frac{\partial x_k}{\partial p_m}\end{bmatrix}.
\end{align}
Now note that
\begin{align}\label{april268}
\textbf{F}(\vec{x},\dot{\vec{x}},p_1,\cdots,p_m)=\vec{0}\,\Longrightarrow\, (D_{\vec{x}}\textbf{F})\frac{\partial \vec{x}}{\partial p_j}+
(D_{\dot{\vec{x}}}\textbf{F})\frac{\partial \dot{\vec{x}}}{\partial p_j}+\frac{\partial \textbf{F}}{\partial p_j}=\vec{0}\,.
\end{align}
That is,
\begin{equation}\label{april264}
(D_{\dot{\vec{x}}}\textbf{F})\dot{\vec{s}}_j+(D_{\vec{x}}\textbf{F})\vec{s}_j+\frac{\partial \textbf{F}}{\partial p_j}=\vec{0}\,.
\end{equation}
So the sensitivities $\vec{s}_j$'s themselves satisfy a system of ODEs (consisting of $k\times m$ equations). The initial or boundary conditions for $\vec{s}_j$ are obtained by taking the partial derivatives of the initial and boundary conditions satisfied by $\vec{x}$ with respect to $p_j$.\\\\
By appending system (\ref{april264}) consisting of $k\times m$ equations to the original system (\ref{april261}), we can construct a system of ODEs with $k\times (m+1)$ equations:
\begin{align}\label{april265}
& \textbf{F}(\vec{x}, \dot{\vec{x}},\vec{p})=\vec{0},\\
&(D_{\dot{\vec{x}}}\textbf{F})\dot{\vec{s}}_j+(D_{\vec{x}}\textbf{F})\vec{s}_j+\frac{\partial \textbf{F}}{\partial p_j}=\vec{0}\,.
\end{align}
Solving the above system with a fixed set of values $\tilde{p}_1,\cdots,\tilde{p}_m$ as parameters and with the appropriate initial and boundary conditions (as described above) gives the original unknowns together with the sensitivities $\displaystyle \vec{s}_j=\frac{\partial \vec{x}}{\partial p_j}$ evaluated at $(\tilde{p}_1,\cdots,\tilde{p}_m)$. If we are only interested in computing sensitivities with respect to one particular parameter $p_j$, it is enough to append the $k$ equations corresponding to $p_j$ to the original system (and hence solving a system with $2k$ equations).\\\\
Now let us focus on the second question. As it was mentioned, we can answer the second question by computing each individual sensitivity using the method explained above (see \Cref{april262}). However, there is at least one more approach that can be used to directly compute the sensitivity of a functional $W$ with respect to the parameters. In this work, we are interested in both sensitivities of solution components and sensitivities of certain functionals of solutions;  hence we will only employ the first approach. Nevertheless, for the sake of completeness, here we briefly describe this alternative approach. In order to explain this second approach, which is sometimes referred to as the ``adjoint sensitivity analysis" \cite{petzold2006}, we need to make a simple observation which we state as a proposition.
\begin{proposition}\label{april266}
Let $\textbf{F}$ and $w$ be as above. Suppose that $\displaystyle \vec{v}(u)=[v_1(u),\ldots ,v_k(u)]^T$ solves the following system of ODEs:
\begin{equation}\label{april267}
\frac{d}{du}\big[\vec{v}^TD_{\dot{\vec{x}}}\textbf{F}\big]-\vec{v}^TD_{\vec{x}}\textbf{F}=-D_{\vec{x}}w.
\end{equation}
Then
\begin{equation}
(D_{\vec{x}}w)\frac{\partial \vec{x}}{\partial p_j}=-\vec{v}^T\frac{\partial \textbf{F}}{\partial p_j}-\frac{d}{du}\big[\vec{v}^TD_{\dot{\vec{x}}}\textbf{F}\frac{\partial \vec{x}}{\partial p_j}\big].
\end{equation}
\end{proposition}
\begin{proof}
We have
\begin{align*}
(D_{\vec{x}}w)\frac{\partial \vec{x}}{\partial p_j}&\stackrel{(\ref{april267})}{=}
\vec{v}^TD_{\vec{x}}\textbf{F}\frac{\partial \vec{x}}{\partial p_j}-
\frac{d}{du}\big[\vec{v}^TD_{\dot{\vec{x}}}\textbf{F}\big]\frac{\partial \vec{x}}{\partial p_j}\\
&=\vec{v}^TD_{\vec{x}}\textbf{F}\frac{\partial \vec{x}}{\partial p_j}-
\frac{d}{du}\big[\vec{v}^TD_{\dot{\vec{x}}}\textbf{F}\frac{\partial \vec{x}}{\partial p_j}\big]+\vec{v}^TD_{\dot{\vec{x}}}\textbf{F}\frac{\partial \dot{\vec{x}}}{\partial p_j}\\
&\stackrel{(\ref{april268})}{=}\vec{v}^T\big[-\frac{\partial \textbf{F}}{\partial p_j}\big]-\frac{d}{du}\big[\vec{v}^TD_{\dot{\vec{x}}}\textbf{F}\frac{\partial \vec{x}}{\partial p_j}\big].
\end{align*}
\end{proof}
It follows that if $\vec{v}(u)$ satisfies (\ref{april267}) then for each $1\leq j\leq m$
\begin{align*}
\frac{\partial W}{\partial p_j}&=\int_0^{u_{max}} \frac{\partial w}{\partial p_j}+(D_{\vec{x}}w)\frac{\partial \vec{x}}{\partial p_j}\,du\\
&= \int_0^{u_{max}}  \frac{\partial w}{\partial p_j}-\vec{v}^T\frac{\partial \textbf{F}}{\partial p_j}-\frac{d}{du}\big[\vec{v}^TD_{\dot{\vec{x}}}\textbf{F}\frac{\partial \vec{x}}{\partial p_j}\big]\,du\\
&=\int_0^{u_{max}} (\frac{\partial w}{\partial p_j}-\vec{v}^T\frac{\partial \textbf{F}}{\partial p_j})\,du-\big[\vec{v}^TD_{\dot{\vec{x}}}\textbf{F}\frac{\partial \vec{x}}{\partial p_j}\big]_{u=0}^{u=u_{max}}.
\end{align*}
Therefore, if we can choose side conditions for $\vec{v}$ such that the boundary term $\displaystyle \big[\vec{v}^TD_{\dot{\vec{x}}}\textbf{F}\frac{\partial \vec{x}}{\partial p_j}\big]_{u=0}^{u=u_{max}}$ vanishes, then we can follow this 2-step process to compute the sensitivities for the functional $W$:
\begin{itemize}
\item \textbf{Step 1:} Solve the following problem (with appropriate side conditions as described above) for the $2k$ unknowns $\vec{x}(u)=[x_1(u),\ldots,x_k(u)]^T$ and $\vec{v}(u)=[v_1(u),\ldots,v_k(u)]^T$:
    \begin{align*}
    \begin{cases}
    &\textbf{F}(\vec{x}, \dot{\vec{x}},\vec{p})=\begin{bmatrix}f_1(\vec{x}, \dot{\vec{x}},\vec{p})\\\vdots\\f_k(\vec{x}, \dot{\vec{x}},\vec{p})\end{bmatrix}=\vec{0},\\
    &\frac{d}{du}\big[\vec{v}^TD_{\dot{\vec{x}}}\textbf{F}\big]-\vec{v}^TD_{\vec{x}}\textbf{F}=-D_{\vec{x}}w.
    \end{cases}
    \end{align*}
\item \textbf{Step 2:} For each $1\leq j\leq m$ compute
\begin{equation*}
\frac{\partial W}{\partial p_j}=\int_0^{u_{max}}  (\frac{\partial w}{\partial p_j}-\vec{v}^T\frac{\partial \textbf{F}}{\partial p_j})\,du.
\end{equation*}
\end{itemize}

\section{Numerical sensitivity analysis of the Helfrich model}\label{sec:numer-results}
In this section, we illustrate the application of the theoretical considerations of sensitivity analysis by applying them to the ODEs of our problem of interest and by conducting numerical experiments for certain parameter choices. In \S \ref{sec:s-framework-for-numerical} we set up a framework for the numerical sensitivity analysis of our boundary value problem. In \Cref{sec:s-nondim-arclength-1,sec:s-nondim-arclength-molly1,sec:s-nondim-arclength-2,sec:s-nondim-arclength-molly2}  
we will discuss some of our numerical results obtained using the dimensionless arclength parametrization. The results obtained using the dimensionless area parametrization, which were consistent with what was observed using the arc-length parametrization, are presented in Appendix A.

\subsection{Framework for numerical sensitivity analysis of the system}\label{sec:s-framework-for-numerical}
In order to facilitate applying the theoretical considerations in \S \ref{sec:over-sensitivity-analysis} to our system of ODEs, we relabel the variables as follows:
\begin{subequations}
\begin{align}
&x_1=x,\quad x_2=y,\quad x_3=\psi,\\
&x_4=h,\quad x_5=l,\quad x_6=\tilde{\lambda}.
\end{align}
\end{subequations}
Using these new labels, we may rewrite our boundary value problem and set up the equations for direct (or adjoint) sensitivity analysis. For each of our numerical experiments, in addition to the choice of parametrization (arc-length parametrization vs. area parametrization), we had to make several other choices including:
\begin{enumerate}
\item Choice of spontaneous curvature function (explained below)
\item Choice of boundary conditions (explained below)
\end{enumerate}
Once we set up the boundary value problem, the MATLAB BVP solver `bvp4c' was used to numerically solve the system. 
Roughly speaking, the domain is partitioned into subintervals and on each subinterval the solution functions are approximated by  polynomials of degree at most 3. 
Note that a cubic polynomial has 4 coefficients, so in order to find the approximate solutions, the solver needs to find 4 coefficients for each unknown function on each subinterval. 
The equations needed to solve for the unknown coefficients are obtained by requiring that each approximate solution must be continuous on the entire interval, and also requiring the differential equation hold at certain points on each subinterval (so the method used by `bvp4c' is in essence a collocation method).

\begin{itemize}
    \item \textbf{Case 1: Dimensionless arc-length parametrization}\\
    Our original boundary value problem can be written as
    \begin{align}\label{may121}
& \textbf{F}(\vec{x},\dot{\vec{x}},C_0,\gamma,t_0)=\vec{0}\\
& \label{eq:bc1}\textrm{(B.C. I)}\,\,\begin{cases}
x_1(0^+)=0,\quad x_3(0^+)=0,\quad x_5(0^+)=0\\
x_2(T)=0,\quad x_3(T)=0,\quad x_6(T)=\tilde{\lambda}_0
\end{cases},\\
& \textrm{or}\notag\\
& \label{eq:bc2}\textrm{(B.C. II)}\,\,\begin{cases}
x_1(0^+)=\sin{\theta},\quad x_3(0^+)=\theta\\
x_2(T)=0,\quad x_3(T)=0,\quad x_5(T)=0,\quad x_6(T)=\tilde{\lambda}_0
\end{cases},
\end{align}
where $\theta$ is a fixed angle and $F$ is as follows:
\begin{align}\label{may122}
\textbf{F}(\vec{x},\dot{\vec{x}},C_0,\gamma,t_0)&=
\begin{bmatrix}
f_1(\vec{x},\dot{\vec{x}},C_0,\gamma,t_0)\\
\vdots\\
f_6(\vec{x},\dot{\vec{x}},C_0,\gamma,t_0)
\end{bmatrix}\notag\\
&=\begin{bmatrix}
\dot{x}_1-\cos{x_3}\\
\dot{x}_2-\sin{x_3}\\
\dot{x}_3-2x_4+\frac{\sin{x_3}}{x_1}\\
\dot{x}_4-\frac{x_5}{x_1}-\dot{c}\\
\dot{x}_5-2x_1x_4\big[(x_4-c)^2+\frac{x_6}{\tilde{\kappa}}\big]+2x_1(x_4-c)\big[x_4^2+(x_4-\frac{\sin{x_3}}{x_1})^2\big]\\
\dot{x}_6-2\tilde{\kappa}\dot{c}x_4+2\tilde{\kappa}c\dot{c}
\end{bmatrix}.
\end{align}
In the above, the function $c(t)$ represents the spontaneous curvature which, in the existing literature, as mentioned before, is chosen to have one of the following forms:
\begin{align}
& \textrm{Type I:}\qquad  c(t)=0.5R_0C_0\big[1-\tanh{[\gamma(t-t_0)]}\big],\\
 & \textrm{Type II:}\qquad c(t)=-0.5R_0C_0\frac{t-t_0}{t_0}\big[1-\tanh{[\gamma(t-t_0)]}\big]. \end{align}
\begin{remark}
Note that the graph of $\tanh$ has two horizontal asymptotes. Roughly, 
\begin{equation}
   \tanh{u}\approx\begin{cases}1\quad &u\geq 3\\-1\quad &u\leq -3\end{cases},
\end{equation}
and so
\begin{equation}
  1-\tanh{[\gamma (t-t_0)]}\approx\begin{cases}0\quad &t\geq t_0+\frac{3}{\gamma}\\2\quad &t\leq t_0-\frac{3}{\gamma}\end{cases},
\end{equation}
Due to this property, we may choose $\gamma$ and $t_0$ so that the graph of $1-\tanh{[\gamma(t-t_0)]}$ possesses an abrupt jump near $t_0$. However, if the parameters $\gamma$ and $t_0$ are chosen such that there is no ``abrupt" jump in the graph of $1-\tanh{[\gamma(t-t_0)]}$ for positive values of $t$ (in particular, this happens if $t_0>>0$ and $\gamma t_0<3$), we say our spontaneous curvature function has a \textbf{mollifying} property. As we shall see, whether or not the spontaneous curvature function is mollifying can affect the solution sensitivities.
\end{remark}

\begin{remark}
An explanation as to why the first set of boundary conditions \Cref{eq:bc1} physically make sense can be found in \cite{hassinger2017}. 
Here is one way to interpret the second set of boundary conditions \Cref{eq:bc2} above: We assume the protein coat consists of two segments; a segment with constant curvature $C_0$ and another segment over which the curvature of the coat decreases from $C_0$ to $0$. 
Clearly, we do not need to consider the constant curvature portion as part of our domain (because the shape of this part is already known). 
The prescribed value of $x_1(0^+)$ should be thought of as the distance from the $x_2$-axis marking the end of the a priori known spherical part of the coated membrane (with constant curvature $C_0$) and start of the part of the membrane whose shape we want to determine (the unknown part of the membrane). By prescribing a value for $x_3(0^+)$ we set the "direction" of the generating curve at the starting point of the unknown part of the membrane. When we use type (II) boundary conditions, our goal is to determine the shape of the part of the membrane that we do not know \textit{a priori}.
\end{remark}
\begin{remark}
Although both types of boundary conditions, \Cref{eq:bc1,eq:bc2}, are taken from the existing literature, it seems to us that each type comes with certain deficiencies that we believe should be clearly discussed.
\begin{itemize}
    \item In the first set of boundary conditions, $x_1(0^+)=0$ and $x_5(0^+)=0$ correspond to $r(0^+)=0$ and $L(0^+)=0$, respectively. However, $L(s)=r(s)[H(s)-C(s)]'$. So it seems that these equalities are not independent. In practice, for all of our simulations the value of $r(0^+)$ is set to be a number close to zero rather than zero, and it is scaled appropriately for the dimensionless parametrizations. This is done primarily in order to avoid numerical difficulties that may arise as a result of the presence of the term $\frac{1}{r(s)}$ in the equations. As a byproduct, this will also take care of the dependency described above. 
    \item In the second set of boundary conditions the side condition $x_5(0^+)=0$ is replaced by $x_5(T)=0$. Hence the above-mentioned dependency does not exist anymore. Nevertheless, it seems to us that the prescribed value of $x_1(0^+)$ is rather arbitrary. As explained above, here the goal was to have a nonzero value for $x_1(0^+)$ marking the end of the spherical part of the coat, however, we could not find any convincing geometric justification as to why $\sin{\theta}$ is a good choice for this nonzero value. 
\end{itemize}
Although we were aware of these issues, in order to make our results relatable to the existing literature on the subject, we decided not to depart from the standard boundary conditions used in the literature.
\end{remark}
The energy functional of interest, that is, the total elastic bending energy of the membrane can be represented by
\begin{align}
W(\vec{x},C_0,\gamma,\alpha_0)\sim\int_0^{T}[x_4(t)-c(t)]^2 x_1(t)\,dt.
\end{align}
Indeed, if we use the change of variable $s=R_0 t$, then
\begin{align}
    \int_0^{T}[h(t)-c(t)]^2x(t)\,dt&=\int_0^{T} R_0^2\big[H(R_0 t)-C(R_0 t)\big]^2\frac{r(R_0t)}{R_0}\,dt\notag\\
    &=\int_0^{R_0T}R_0 [H(s)-C(s)]^2r(s)\,\frac{d s}{R_0}\notag\\
    &=\int_0^S[H(s)-C(s)]^2r(s)\,ds\sim W.
\end{align}
In order to ensure that our notations are consistent with what was discussed in \S \ref{sec:over-sensitivity-analysis}, we also introduce $p_1$, $p_2$, and $p_3$ as follows:
\begin{equation}
p_1=C_0,\quad p_2=\gamma,\quad  p_3=t_0.
\end{equation}
Following what was discussed in \S \ref{sec:over-sensitivity-analysis}, for each $1\leq j\leq 3$, in order to find the sensitivities $\displaystyle \vec{s}_j=\frac{\partial \vec{x}}{\partial p_j}$, it is enough to solve the following system of ODEs for the unknowns $\vec{x}(t)$ and $\vec{s}_j(t)$:
\begin{align}
\begin{cases}
\qquad \qquad \textbf{F}(\vec{x},\dot{\vec{x}},p_1,p_2,p_3)&=\vec{0}\qquad (\textrm{six scalar ODEs})\\
(D_{\dot{\vec{x}}}\textbf{F})\dot{\vec{s}}_j+(D_{\vec{x}}\textbf{F})\vec{s}_j+\frac{\partial \textbf{F}}{\partial p_j}&=\vec{0}\qquad (\textrm{six scalar ODEs})
\end{cases}.
\end{align}
The boundary conditions for the sensitivities can be obtained by taking the derivative of the boundary conditions for the original unknowns $x_1,\cdots,x_6$. For example, if we decide to use the first type of boundary conditions, then we will equip the above system of ODEs with the following boundary conditions:
\begin{align}
\begin{cases}
&x_1(0^+)=0,\quad x_3(0^+)=0,\quad x_5(0^+)=0,\\
&x_2(T)=0,\quad x_3(T)=0,\quad x_6(T)=\tilde{\lambda}_0,\\
&s_{j1}(0^+)=0,\quad s_{j3}(0^+)=0,\quad s_{j5}(0^+)=0,\\ &s_{j2}(T)=0,\quad s_{j3}(T)=0,\quad s_{j6}(T)=0\,.
\end{cases}
\end{align}
In the above $s_{ji}$ denotes the $i^{th}$ component of the vector $\vec{s}_j$, that is, $\displaystyle s_{ji}(t)=\frac{\partial x_i(t)}{\partial p_j}$. Once we find each $\vec{s}_j$, we can use the following formula to compute the sensitivities of our energy functional:
\begin{equation}
\forall\, 1\leq j\leq 3\qquad \frac{\partial W}{\partial p_j}\sim\int_0^{T}\frac{\partial w}{\partial p_j}+(D_{\vec{x}}w)\vec{s}_j\,dt.
\end{equation}
where $w(\vec{x},p_1,p_2,p_3)=[x_4(t)-c(t)]^2x_1(t)$. Sometimes we denote the integrand in the above integral by $\frac{D w}{D p_j}$.
In order to set up the above system and also compute the sensitivities of the energy functional $W$, in particular we need to compute the matrices $D_{\dot{\vec{x}}}\textbf{F}$, $D_{\vec{x}}\textbf{F}$ and the vector $\frac{\partial \textbf{F}}{\partial p_j}$ for each $1\leq j\leq 3$. The details can be found in Appendix \ref{app:derivatives}.
\item \textbf{Case 2: Dimensionless area parametrization}\\ 
Our original boundary value problem can be written as
\begin{align}\label{april269}
& \textbf{F}(\vec{x},\dot{\vec{x}},C_0,\gamma,\alpha_0)=\vec{0}\\
& \label{eqn-area-bc-1}\textrm{(B.C. I)}\,\,\begin{cases}
x_1(0^+)=0,\quad x_3(0^+)=0,\quad x_5(0^+)=0\\
x_2(\alpha_{max})=0,\quad x_3(\alpha_{max})=0,\quad x_6(\alpha_{max})=\tilde{\lambda}_0
\end{cases},\\
& \textrm{or}\notag\\
& \label{eqn-area-bc-2}\textrm{(B.C. II)}\,\,\begin{cases}
x_1(0^+)=\sin{\theta},\quad x_3(0^+)=\theta\\
x_2(\alpha_{max})=0,\quad x_3(\alpha_{max})=0,\quad x_5(\alpha_{max})=0,\quad x_6(\alpha_{max})=\tilde{\lambda}_0
\end{cases},
\end{align}

where $\theta$ is a fixed angle and $F$ is as follows:
\begin{align}\label{april2610}
\textbf{F}(\vec{x},\dot{\vec{x}},C_0,\gamma,\alpha_0)&=
\begin{bmatrix}
f_1(\vec{x},\dot{\vec{x}},C_0,\gamma,\alpha_0)\\
\vdots\\
f_6(\vec{x},\dot{\vec{x}},C_0,\gamma,\alpha_0)
\end{bmatrix}\notag\\
&=\begin{bmatrix}
\dot{x}_1-\frac{\cos{x_3}}{x_1}\\
\dot{x}_2-\frac{\sin{x_3}}{x_1}\\
\dot{x}_3-2\frac{x_4}{x_1}+\frac{\sin{x_3}}{x_1^2}\\
\dot{x}_4-\frac{x_5}{x_1^2}-\dot{c}\\
\dot{x}_5-2x_4\big[(x_4-c)^2+\frac{x_6}{\tilde{\kappa}}\big]+2(x_4-c)\big[x_4^2+(x_4-\frac{\sin{x_3}}{x_1})^2\big]\\
\dot{x}_6-2\tilde{\kappa}\dot{c}x_4+2\tilde{\kappa}c\dot{c}
\end{bmatrix}.
\end{align}
In the above, the function $c(\alpha)$ represents the spontaneous curvature which, as was discussed in previous sections, is chosen to have one of the following forms:
\begin{align}
& \textrm{Type I:}\qquad  c(\alpha)=0.5R_0C_0\big[1-\tanh{[\gamma(\alpha-\alpha_0)]}\big],\\
 & \textrm{Type II:}\qquad c(\alpha)=-0.5R_0C_0\frac{\alpha-\alpha_0}{\alpha_0}\big[1-\tanh{[\gamma(\alpha-\alpha_0)]}\big].   
\end{align}
The energy functional of interest, that is, the total elastic bending energy of the membrane can be represented by
\begin{align}
W(\vec{x},C_0,\gamma,\alpha_0)\sim\int_0^{\alpha_{max}}[x_4(\alpha)-c(\alpha)]^2\,d\alpha.
\end{align}
Indeed, if we use the change of variable $a=2\pi R_0^2\alpha$, then
\begin{align}
    \int_0^{\alpha_{max}}[h(\alpha)-c(\alpha)]^2\,d\alpha&=\int_0^{\alpha_{max}} R_0^2\big[H(2\pi R_0^2\alpha)-C(2\pi R_0^2\alpha)\big]^2\,d\alpha\notag\\
    &=\int_0^{2\pi R_0^2\alpha_{max}}R_0^2 [H(a)-C(a)]^2\,\frac{d a}{2\pi R_0^2}\notag\\
    &=\frac{1}{2\pi}\int_0^A[H(a)-C(a)]^2\,da\sim W.
\end{align}
Again, in order to ensure that our notations are consistent with what was discussed in \S \ref{sec:over-sensitivity-analysis}, we also introduce $p_1$, $p_2$, and $p_3$ as follows:
\begin{equation}
p_1=C_0,\quad p_2=\gamma,\quad  p_3=\alpha_0.
\end{equation}
Following what was discussed in \S \ref{sec:over-sensitivity-analysis}, for each $1\leq j\leq 3$, in order to find the sensitivities $\displaystyle \vec{s}_j=\frac{\partial \vec{x}}{\partial p_j}$, it is enough to solve the following system of ODEs for the unknowns $\vec{x}(\alpha)$ and $\vec{s}_j(\alpha)$:
\begin{align}
\begin{cases}
\qquad \qquad \textbf{F}(\vec{x},\dot{\vec{x}},p_1,p_2,p_3)&=\vec{0} \qquad (\textrm{six scalar ODEs})\\
(D_{\dot{\vec{x}}}\textbf{F})\dot{\vec{s}}_j+(D_{\vec{x}}\textbf{F})\vec{s}_j+\frac{\partial \textbf{F}}{\partial p_j}&=\vec{0} \qquad (\textrm{six scalar ODEs})
\end{cases}.
\end{align}
The boundary conditions for the sensitivities can be obtained by taking the derivative of the boundary conditions for the original unknowns $x_1,\cdots,x_6$. For example, if we decide to use the second type of boundary conditions, then we will equip the above system of ODEs with the following boundary conditions:
\begin{align}
\begin{cases}
&x_1(0^+)=\sin{\theta},\quad x_3(0^+)=\theta,\\
&x_2(\alpha_{max})=0,\quad x_3(\alpha_{max})=0,\quad x_5(\alpha_{max})=0,\quad x_6(\alpha_{max})=\tilde{\lambda}_0,\\
&s_{j1}(0^+)=0,\quad s_{j3}(0^+)=0,\\ &s_{j2}(\alpha_{max})=0,\quad s_{j3}(\alpha_{max})=0,\quad s_{j5}(\alpha_{max})=0,\quad s_{j6}(\alpha_{max})=0\,.
\end{cases}
\end{align}
As before $s_{ji}$ denotes the $i^{th}$ component of the vector $\vec{s}_j$, that is, $\displaystyle s_{ji}(\alpha)=\frac{\partial x_i(\alpha)}{\partial p_j}$. Once we find each $\vec{s}_j$, we can use the following formula to compute the sensitivities of our energy functional:
\begin{equation}
\forall\, 1\leq j\leq 3\qquad \frac{\partial W}{\partial p_j}\sim\int_0^{\alpha_{max}}\frac{\partial w}{\partial p_j}+(D_{\vec{x}}w)\vec{s}_j\,d \alpha.
\end{equation}
where $w(\vec{x},p_1,p_2,p_3)=[x_4(\alpha)-c(\alpha)]^2$. The integrand in the above integral is sometimes denoted by $\frac{D w}{D p_j}$. 
In order to set up the above system and also compute the sensitivities of the energy functional $W$, in particular we need to compute the matrices $D_{\dot{\vec{x}}}\textbf{F}$, $D_{\vec{x}}\textbf{F}$ and the vector  $\frac{\partial \textbf{F}}{\partial p_j}$ for each $1\leq j\leq 3$. The details can be found in Appendix \ref{app:derivatives}.
\end{itemize}
\begin{remark}
Alternatively, according to the adjoint method, in order to compute the sensitivities of the energy functional $W$, we may use the following formula:
\begin{equation}
\frac{\partial W}{\partial p_j}\sim\int_0^{u_{max}}  (\frac{\partial w}{\partial p_j}-\vec{v}^T\frac{\partial \textbf{F}}{\partial p_j})\,du,
\end{equation}
where, for example, in the case where we use the dimensionless area parametrization with boundary conditions of type (I), $u_{max}$ and $u$ in the above formula stand for $\alpha$ and $\alpha_{max}$, respectively, and the vector function  $\displaystyle \vec{v}=\begin{bmatrix}v_1\\\vdots\\v_6\end{bmatrix}$ can be computed by finding a solution of the following system of ODEs (with 12 scalar unknown functions):
\begin{align*}
\begin{cases}
&\textbf{F}(\vec{x},\dot{\vec{x}},p_1,p_2,p_3)=\vec{0},\\
& \frac{d}{d\alpha}\big[\vec{v}^TD_{\dot{\vec{x}}}\textbf{F}\big]-\vec{v}^TD_{\vec{x}}\textbf{F}=-D_{\vec{x}}w,\\
&
x_1(0^+)=0,\quad x_3(0^+)=0,\quad x_5(0^+)=0,\\
&x_2(\alpha_{max})=0,\quad x_3(\alpha_{max})=0,\quad x_6(\alpha_{max})=\tilde{\lambda}_0,\\
&
v_2(0^+)=0,\quad v_4(0^+)=0,\quad v_6(0^+)=0,\\
&v_1(\alpha_{max})=0,\quad v_4(\alpha_{max})=0,\quad v_5(\alpha_{max})=0.
\end{cases}
\end{align*}
The following proposition justifies the particular choice of boundary conditions for certain components of $\vec{v}$ in the above boundary value problem.
\begin{proposition}
Let $\vec{x}$ and $\textbf{F}$ be as above. Suppose that
\begin{align*}
&
x_1(0^+)=0,\quad x_3(0^+)=0,\quad x_5(0^+)=0,\\
&x_2(\alpha_{max})=0,\quad x_3(\alpha_{max})=0,\quad x_6(\alpha_{max})=\tilde{\lambda}_0.
\end{align*} 
If $\displaystyle \vec{v}(\alpha)=\begin{bmatrix}v_1(\alpha)\\\vdots\\v_6(\alpha)\end{bmatrix}$ satisfies
\begin{align*}
&
v_2(0^+)=0,\quad v_4(0^+)=0,\quad v_6(0^+)=0,\\
&v_1(\alpha_{max})=0,\quad v_4(\alpha_{max})=0,\quad v_5(\alpha_{max})=0.
\end{align*} 
then $\big[\vec{v}^TD_{\dot{\vec{x}}}\textbf{F}\frac{\partial \vec{x}}{\partial p_j}\big]_{\alpha=0}^{\alpha=\alpha_{max}}=0$.
\end{proposition} 
\begin{proof}
At $\alpha=0^+$ we have
\begin{align*}
\vec{v}^T(0^+)D_{\dot{\vec{x}}}\textbf{F}|_{\alpha=0^+}\frac{\partial \vec{x}(0^+)}{\partial p_j}&=
\begin{bmatrix}
v_1(0^+)& \cdots& v_6(0^+)
\end{bmatrix}
\begin{bmatrix}
\frac{\partial \textbf{F}}{\partial \dot{x}_1}|_{\alpha=0^+}&\cdots& \frac{\partial \textbf{F}}{\partial \dot{x}_6}|_{\alpha=0^+}
\end{bmatrix}
\begin{bmatrix}
0\\
\frac{x_2(0^+)}{\partial p_j}\\
0\\
\frac{x_4(0^+)}{\partial p_j}\\
0\\
\frac{x_6(0^+)}{\partial p_j}
\end{bmatrix}\\
&\hspace{-1.2cm}=\begin{bmatrix}
v_1(0^+)& \cdots& v_6(0^+)
\end{bmatrix}
\bigg(\frac{\partial x_2(0^+)}{\partial p_j} \frac{\partial \textbf{F}}{\partial \dot{x}_2}|_{\alpha=0^+}
+\frac{\partial x_4(0^+)}{\partial p_j} \frac{\partial \textbf{F}}{\partial \dot{x}_4}|_{\alpha=0^+}
+\frac{\partial x_6(0^+)}{\partial p_j} \frac{\partial \textbf{F}}{\partial \dot{x}_6}|_{\alpha=0^+}\bigg)\\
&\hspace{-1.2cm}=\frac{\partial x_2(0^+)}{\partial p_j} v_2(0^+)+\frac{\partial x_4(0^+)}{\partial p_j} v_4(0^+)+\frac{\partial x_6(0^+)}{\partial p_j} v_6(0^+)=0.
\end{align*}
In the above, we used $v_2(0^+)=v_4(0^+)=v_6(0^+)=0$ and the fact that $D_{\dot{\vec{x}}}\textbf{F}$ is the identity matrix.\\\\
Similarly, at $\alpha=\alpha_{max}$, we have
\begin{align*}
&\vec{v}^T(\alpha_{max})D_{\dot{\vec{x}}}\textbf{F}|_{\alpha=\alpha_{max}}\frac{\partial \vec{x}(\alpha_{max})}{\partial p_j}\\
&=
\begin{bmatrix}
v_1(\alpha_{max})& \cdots& v_6(\alpha_{max})
\end{bmatrix}
\begin{bmatrix}
\frac{\partial \textbf{F}}{\partial \dot{x}_1}|_{\alpha=\alpha_{max}}&\cdots& \frac{\partial \textbf{F}}{\partial \dot{x}_6}|_{\alpha=\alpha_{max}}
\end{bmatrix}
\begin{bmatrix}
\frac{x_1(\alpha_{max})}{\partial p_j}\\
0\\
0\\
\frac{x_4(\alpha_{max})}{\partial p_j}\\
\frac{x_5(\alpha_{max})}{\partial p_j}\\
0
\end{bmatrix}\\
&=\begin{bmatrix}
v_1(\alpha_{max})& \cdots& v_6(\alpha_{max})
\end{bmatrix}
\bigg(\frac{\partial x_1(\alpha_{max})}{\partial p_j} \frac{\partial \textbf{F}}{\partial \dot{x}_1}|_{\alpha=\alpha_{max}}
+\frac{\partial x_4(\alpha_{max})}{\partial p_j} \frac{\partial \textbf{F}}{\partial \dot{x}_4}|_{\alpha=\alpha_{max}}
\\&\hspace{6cm}+\frac{\partial x_5(\alpha_{max})}{\partial p_j} \frac{\partial \textbf{F}}{\partial \dot{x}_5}|_{\alpha=\alpha_{max}}
\bigg)\\
&=\frac{\partial x_1(\alpha_{max})}{\partial p_j} v_1(\alpha_{max})+\frac{\partial x_4(\alpha_{max})}{\partial p_j} v_4(\alpha_{max})+\frac{\partial x_5(\alpha_{max})}{\partial p_j} v_5(\alpha_{max})=0.
\end{align*}
In the above, we used $v_1(\alpha_{max})=v_4(\alpha_{max})=v_5(\alpha_{max})=0$ and the fact that $D_{\dot{\vec{x}}}\textbf{F}$ is the identity matrix. 
\end{proof}
\end{remark}

\subsection{Dimensionless arc-length formulation -- Type I spontaneous curvature}\label{sec:s-nondim-arclength-1}
    
    The spontaneous curvature function used is:
    \begin{equation}
    c(t)=0.5C_0R_0\big[1-\tanh\big[\gamma(t-t0)\big]\big].
    \end{equation}
    Note that, for each choice of $C_0$, $c(t)$ is a constant multiple of $F(t)=1-\tanh\big[\gamma(t-t_0)\big]$. So $F(t)$ is the function that determines the shape of the spontaneous curvature. The graph of $F(t)$ is depicted in  \Cref{fig:Figure2}\\\\
    Boundary conditions used for the original unknowns are:
\begin{align}
\begin{cases}
&x_1(0^+)=0,\quad x_3(0^+)=0,\quad x_5(0^+)=0,\\
&x_2(T)=0,\quad x_3(T)=0,\quad x_6(T)=\tilde{\lambda}_0.
\end{cases}
\end{align}
  The boundary conditions for the corresponding sensitivities are all set to be zero.\\
  Our choices for the input parameters are given in \Cref{tab:Table 1}.
    \begin{table}[H]
    \centering
    \begin{tabular}{| >{\centering\arraybackslash}m{2in} | >{\centering\arraybackslash}m{2in} |@{}m{0cm}@{}}
\hline Parameter & Input used\\
\hline $R_0$ $(nm)$ & $80$&\\ 
\hline $C_0$ $(nm^{-1})$ & various values in $[0,0.02]$&\\
\hline $\gamma=\xi*R_0$ & $20$&\\
\hline $t_0=\frac{s_0}{R_0}$ & $1$&\\
\hline $\tilde{\lambda}_0=\lambda_0R_0^2/\kappa_0$  & $12.8$&\\
\hline Domain $=[0,T]$ & $[0,5]$&\\
\hline
\end{tabular}
\caption{Parameters used in the model. The corresponding diagrams are depicted in \Cref{fig:Figure2}.} 
\label{tab:Table 1}
\end{table}
\noindent $R_0$ is chosen such that $t_0=1$. A brief explanation of the choice $s_0=80$ and $T=5$ is given in \S \ref{sec:conclusion}. 
$\gamma$ is chosen such that there will be a sharp transition at $t_0=1$. $\kappa_0$ is chosen to be equal to $\kappa$. (A range of acceptable values for $\kappa$ and $\lambda_0$ can be found in \cite{hassinger2017}.)\\

We organize our results as follows: graphs of energy sensitivities with respect to parameters $C_0$, $\lambda$, and $t_0$ are displayed on the left column of \Cref{fig:Figure2}B, \Cref{fig:Figure2}C, and \Cref{fig:Figure2}D, respectively; the graphs of the dimensionless curvature sensitivities with respect to parameters involved are depicted on the right column of these panels.
 
 Here we make two key observations. 
 First, we notice that for our choice of the type of boundary conditions (B.C. of type I, \Cref{eq:bc1}) and spontaneous curvature function, the energy sensitivity graphs do not intersect the horizontal axis (that is, there is no critical point). 
 As we shall see later, this seems to be in correlation with certain interesting properties of the final shape of the membrane. 
 Second, notice that for our choice of boundary conditions (B.C. of type I, \Cref{eq:bc1}) and spontaneous curvature function (which possesses a sharp transition), there is an abrupt change in curvature sensitivities near the location where the sharp transition in spontaneous curvature function occurs, which, of course, is expected. 
 As we shall soon see, the graphs of curvature sensitivities will smear out if we smooth the transition in the spontaneous curvature function. 

\begin{figure}[H]
    \centering
    \includegraphics[height=5in,width=0.7\textwidth]{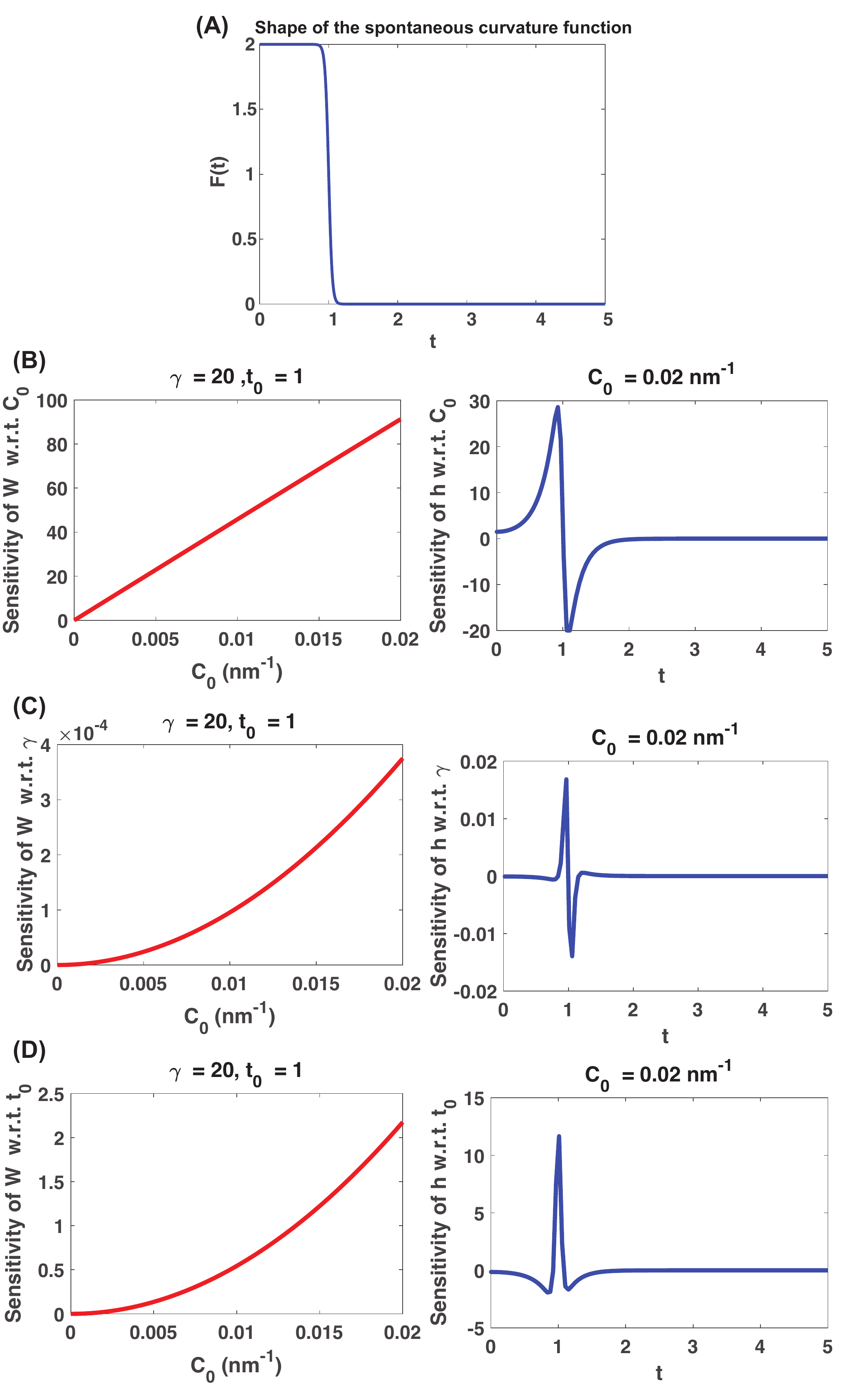}
    \caption{(A) Shape of the spontaneous curvature function with $\gamma=20$ and $t_0=1$.  (B) The red curve on the left depicts $\frac{\partial W}{\partial C_0}|_{(C_0,\gamma,t_0)=(C_0,20,1)}$ for $0\leq C_0\leq 0.02$. The blue curve on the right depicts $\frac{\partial h}{\partial C_0}|_{(C_0,\gamma,t_0)=(0.02,20,1)}$ for $0\leq t\leq 5$. (C) The red curve on the left depicts $\frac{\partial W}{\partial \gamma}|_{(C_0,\gamma,t_0)=(C_0,20,1)}$ for $0\leq C_0\leq 0.02$. The blue curve on the right depicts $\frac{\partial h}{\partial \gamma}|_{(C_0,\gamma,t_0)=(0.02,20,1)}$ for $0\leq t\leq 5$. (D) The red curve on the left depicts $\frac{\partial W}{\partial t_0}|_{(C_0,\gamma,t_0)=(C_0,20,1)}$ for $0\leq C_0\leq 0.02$. The blue curve on the right depicts $\frac{\partial h}{\partial t_0}|_{(C_0,\gamma,t_0)=(0.02,20,1)}$ for $0\leq t\leq 5$. Note that $W$ is a scalar multiple of the total elastic bending energy of the membrane and $h(t)$ is a scalar multiple of the mean curvature at the corresponding points on the membrane.}
    \label{fig:Figure2}
\end{figure}


\subsection{Dimensionless arc-length formulation -- Mollifying type I spontaneous curvature}\label{sec:s-nondim-arclength-molly1}

All parameters are exactly the same as \S \ref{sec:s-nondim-arclength-1} except $\gamma$. 
This time we choose $\gamma$ to be equal to $2.5$ so that $s_0\gamma<3$. 
This choice of $\gamma$ and $s_0$ smooths the transition in the spontaneous curvature function (see \Cref{fig:Figure3}A). As a result we will see that the graphs of curvature sensitivities (diagrams in the second column of \Cref{fig:Figure3}B, C, and D) are smeared out (compared with what was observed in \S \ref{sec:s-nondim-arclength-1}). Notice that this particular change in the shape of the spontaneous curvature function did not have any effect on the number of horizontal intercepts of the energy sensitivities.

\begin{figure}[H]
    \centering
    \includegraphics[height=5in,width=0.7\textwidth]{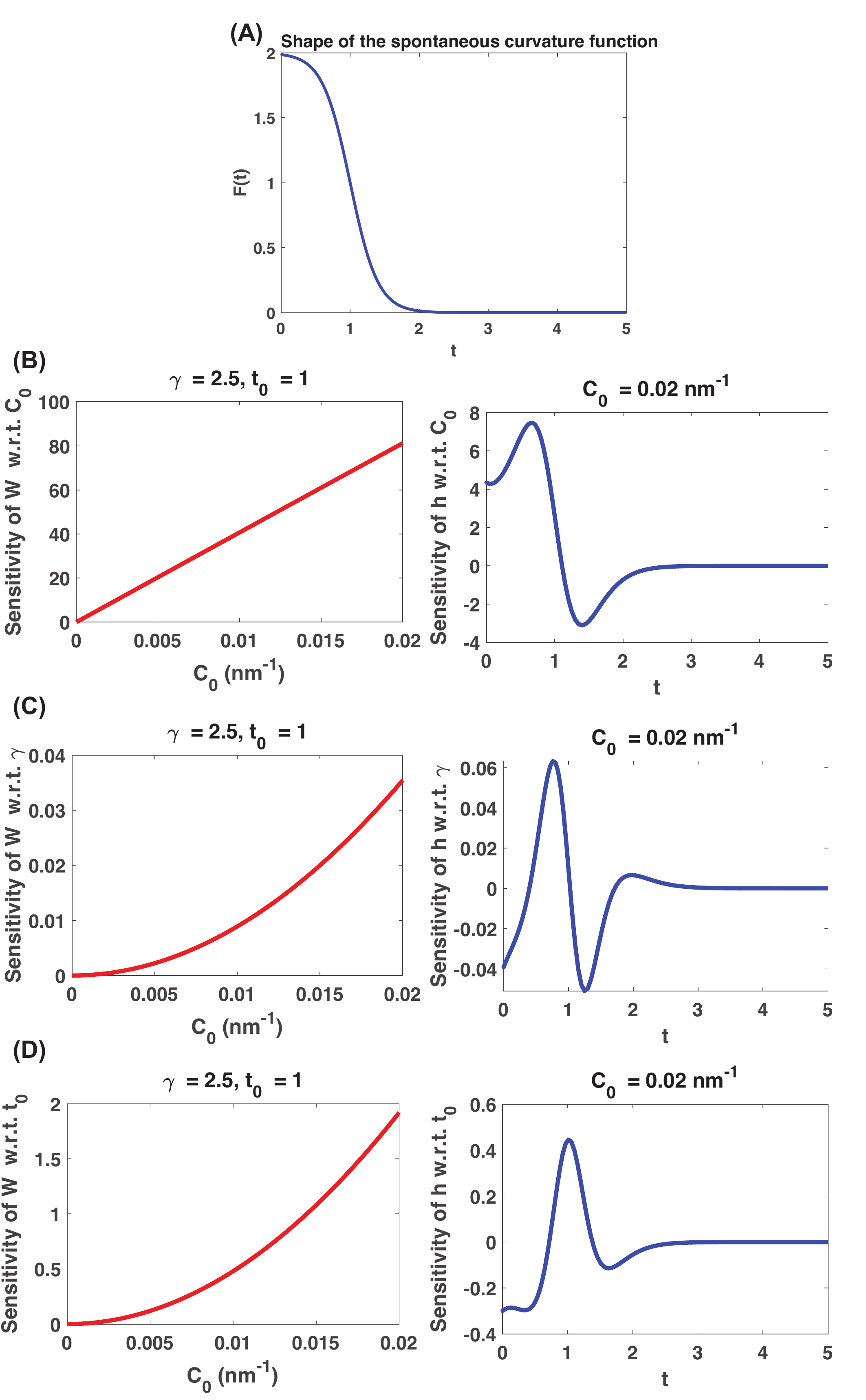}
    \caption{(A) Shape of the spontaneous curvature function with $\gamma=2.5$ and $t_0=1$.  (B) The red curve on the left depicts $\frac{\partial W}{\partial C_0}|_{(C_0,\gamma,t_0)=(C_0,2.5,1)}$ for $0\leq C_0\leq 0.02$. The blue curve on the right depicts $\frac{\partial h}{\partial C_0}|_{(C_0,\gamma,t_0)=(0.02,2.5,1)}$ for $0\leq t\leq 5$. (C) The red curve on the left depicts $\frac{\partial W}{\partial \gamma}|_{(C_0,\gamma,t_0)=(C_0,2.5,1)}$ for $0\leq C_0\leq 0.02$. The blue curve on the right depicts $\frac{\partial h}{\partial \gamma}|_{(C_0,\gamma,t_0)=(0.02,2.5,1)}$ for $0\leq t\leq 5$. (D) The red curve on the left depicts $\frac{\partial W}{\partial t_0}|_{(C_0,\gamma,t_0)=(C_0,2.5,1)}$ for $0\leq C_0\leq 0.02$. The blue curve on the right depicts $\frac{\partial h}{\partial t_0}|_{(C_0,\gamma,t_0)=(0.02,2.5,1)}$ for $0\leq t\leq 5$. Note that $W$ is a scalar multiple of the total elastic bending energy of the membrane and $h(t)$ is a scalar multiple of the mean curvature at the corresponding points on the membrane.}
    \label{fig:Figure3}
\end{figure}


\subsection{Dimensionless arc-length formulation -- Type II spontaneous curvature}\label{sec:s-nondim-arclength-2}
The spontaneous curvature function used is:
    \begin{equation}
    c(t)=-0.5C_0R_0\frac{t-t_0}{t_0}\big[1-\tanh\big[\gamma(t-t0)\big]\big].
    \end{equation}
    Note that for each choice of $C_0$, $c(t)$ is a constant multiple of $F(t)=-\frac{t-t_0}{t_0}\big[1-\tanh{[\gamma(t-t_0)]}\big]$. The graph of $F(t)$ is depicted in  \Cref{fig:Figure6}A. Boundary conditions used for the original unknowns are:
\begin{align}
\begin{cases}
&x_1(0^+)=\sin{(0.9\pi)},\quad  x_3(0^+)=0.9\pi,\\
&x_2(T)=0,\quad x_3(T)=0,\quad x_5(T)=0,\quad x_6(T)=0.
\end{cases}
\end{align}
  The boundary conditions for the corresponding sensitivities are all set to be zero.\\
  Our choices for the input parameters are given in \Cref{tab:Table 3}.
    \begin{table}[H]
    \centering
    \begin{tabular}{| >{\centering\arraybackslash}m{2in} | >{\centering\arraybackslash}m{2in} |@{}m{0cm}@{}}
\hline Parameter & Input used\\
\hline $R_0$ $(nm)$ & $200$&\\ 
\hline $C_0$ $(nm^{-1})$ & various values in $[0, \frac{4.5}{1000}]$ &\\
\hline $\gamma=\xi R_0$ & $30$&\\
\hline $t_0=\frac{s_0}{R_0}$ & $30$&\\
\hline Domain $=[0,T]$ & $[0,100]$&\\
\hline
\end{tabular}
\caption{Parameters used in the model. The corresponding diagrams are depicted in \Cref{fig:Figure6}.}
\label{tab:Table 3}
\end{table}
\noindent This particular choice of parameters is in agreement with \cite{yuan2020}.

We organize our results as follows: graphs of energy sensitivities with respect to parameters $C_0$, $\lambda$, and $t_0$ are displayed on the left column of \Cref{fig:Figure6}B, C, and D, respectively; the graphs of the dimensionless curvature sensitivities with respect to parameters involved are depicted on the right column of these panels. We make two observations. First, notice that for our choice of the type of boundary conditions (B.C. of type II, \Cref{eq:bc2}) and spontaneous curvature function, the energy sensitivity graph with respect to $C_0$ indeed intersects the horizontal axis. As we shall see, this seems to be in correlation with certain interesting properties of the final shape of the membrane. Second, notice that for our choice of boundary conditions and spontaneous curvature function (which possesses a sharp transition), the graphs of curvature sensitivities have sharp bends near the location where the sharp transition in spontaneous curvature function occurs. As we shall soon see, the graphs of curvature sensitivities will be smoother near $t_0$ if we smooth the sharp transition in the spontaneous curvature function.

\begin{figure}[H]
    \centering
    \includegraphics[height=5in,width=0.7\textwidth]{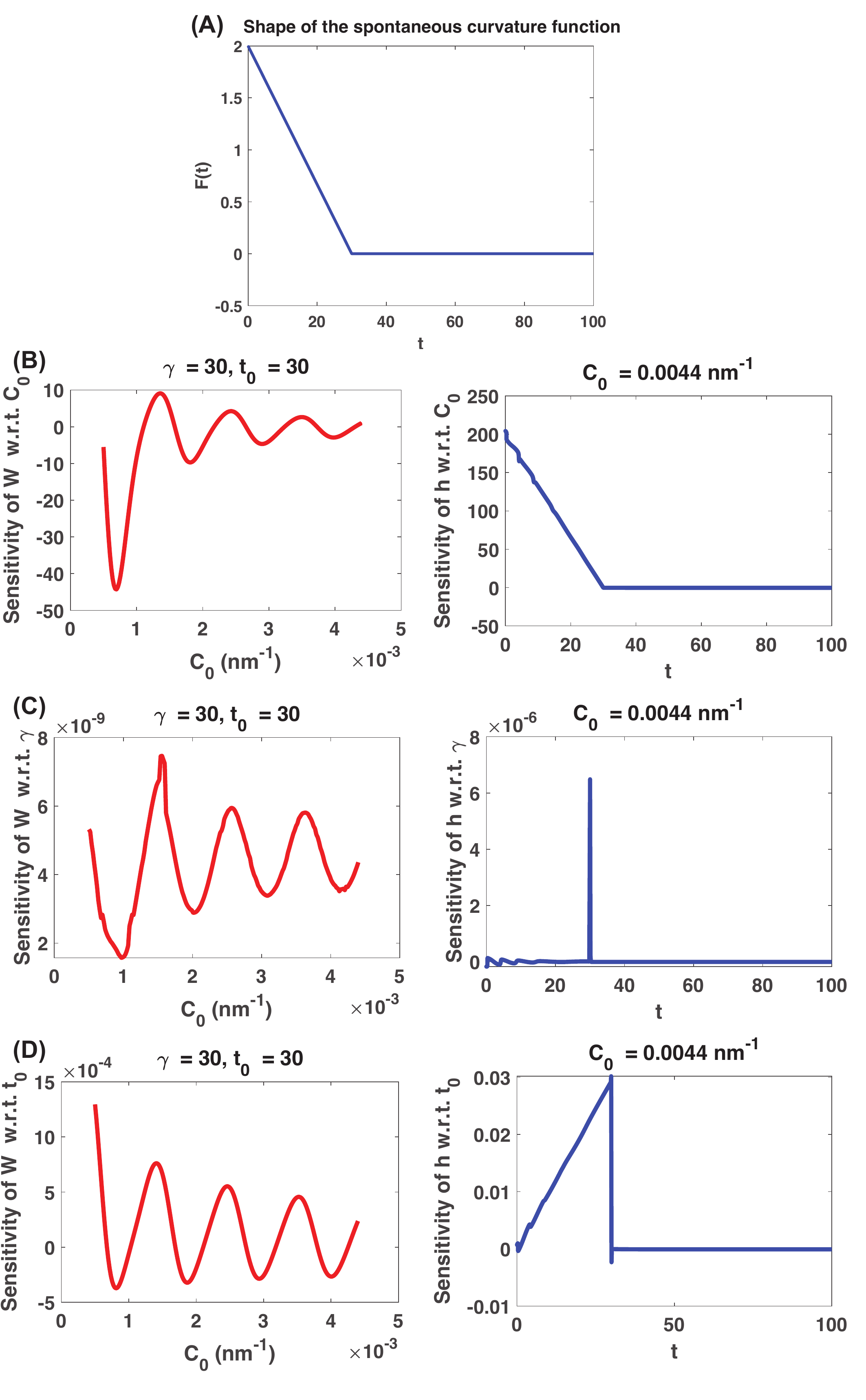}
    \caption{(A) Shape of the spontaneous curvature function with $\gamma=30$ and $t_0=30$.  (B) The red curve on the left depicts $\frac{\partial W}{\partial C_0}|_{(C_0,\gamma,t_0)=(C_0,30,30)}$ for $0\leq C_0\leq 0.0044$. The blue curve on the right depicts $\frac{\partial h}{\partial C_0}|_{(C_0,\gamma,t_0)=(0.0044,30,30)}$ for $0\leq t\leq 100$. (C) The red curve on the left depicts $\frac{\partial W}{\partial \gamma}|_{(C_0,\gamma,t_0)=(C_0,30,30)}$ for $0\leq C_0\leq 0.0044$. The blue curve on the right depicts $\frac{\partial h}{\partial \gamma}|_{(C_0,\gamma,t_0)=(0.0044,30,30)}$ for $0\leq t\leq 100$. (D) The red curve on the left depicts $\frac{\partial W}{\partial t_0}|_{(C_0,\gamma,t_0)=(C_0,30,30)}$ for $0\leq C_0\leq 0.0044$. The blue curve on the right depicts $\frac{\partial h}{\partial t_0}|_{(C_0,\gamma,t_0)=(0.0044,30,30)}$ for $0\leq t\leq 100$. Note that $W$ is a scalar multiple of the total elastic bending energy of the membrane and $h(t)$ is a scalar multiple of the mean curvature at the corresponding points on the membrane.}
    \label{fig:Figure6}
\end{figure}

\subsection{Dimensionless arc-length formulation -- Mollifying type II spontaneous curvature}\label{sec:s-nondim-arclength-molly2}

Almost all parameters are the same those used in \S \ref{sec:s-nondim-arclength-2} except for the value chosen for $\gamma$. 
This time we choose $\gamma$ to be equal to $0.08$ such that $s_0\gamma<3$. This choice of $\gamma$ and $s_0$ further smooths the transition in the spontaneous curvature function (see \Cref{fig:Figure8}A). As a result we will see that the graphs of curvature sensitivities (diagrams in the second column of \Cref{fig:Figure8}B, C, and D) are smoother near $t_0=30$ (compared with what was observed in \S \ref{sec:s-nondim-arclength-2}). 
\begin{figure}[H]
    \centering
    \includegraphics[height=5in,width=0.7\textwidth]{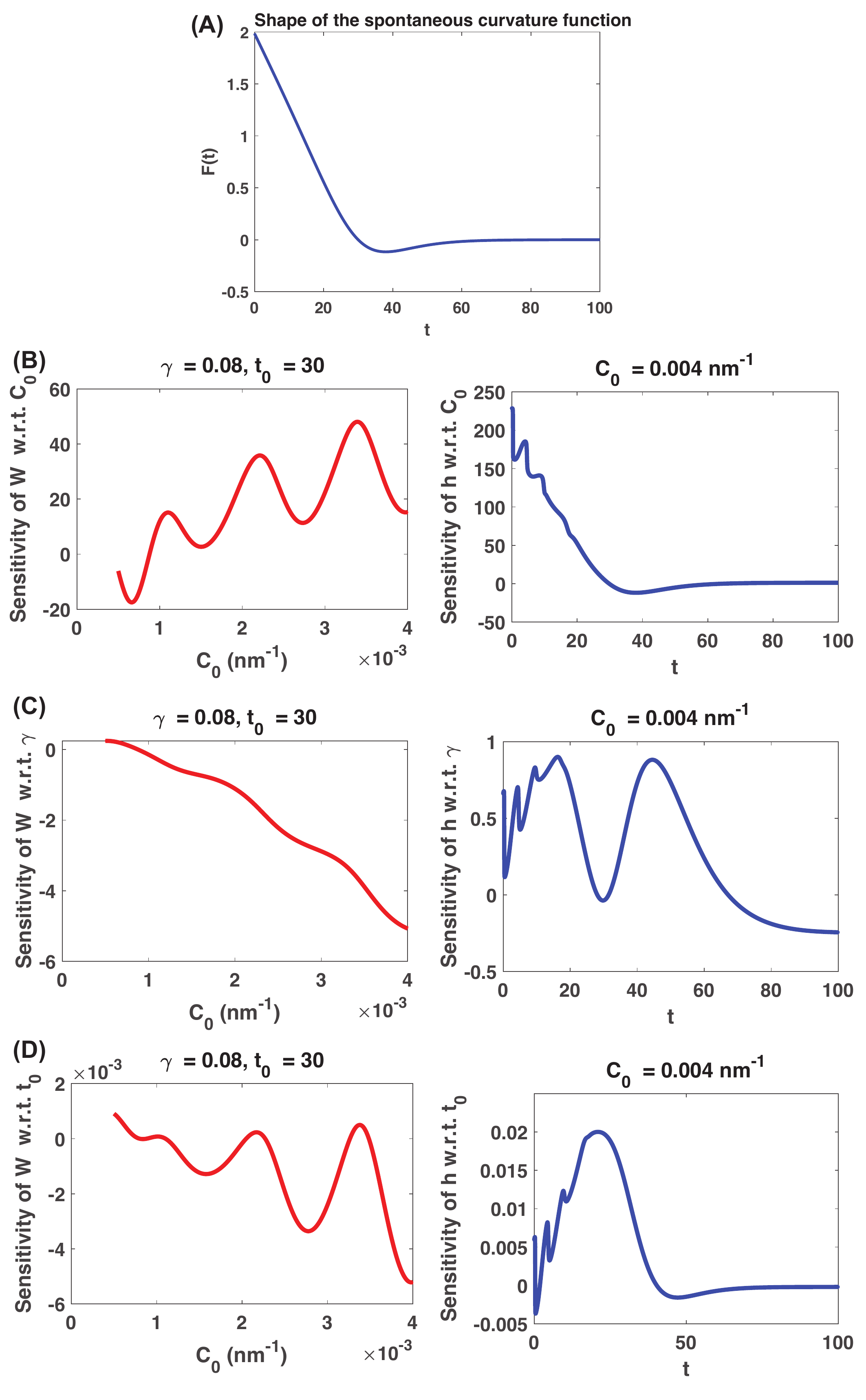}
    \caption{(A) Shape of the spontaneous curvature function with $\gamma=0.08$ and $t_0=30$.  (B) The red curve on the left depicts $\frac{\partial W}{\partial C_0}|_{(C_0,\gamma,t_0)=(C_0,0.08,30)}$ for $0\leq C_0\leq 0.004$. The blue curve on the right depicts $\frac{\partial h}{\partial C_0}|_{(C_0,\gamma,t_0)=(0.004,0.08,30)}$ for $0\leq t\leq 100$. (C) The red curve on the left depicts $\frac{\partial W}{\partial \gamma}|_{(C_0,\gamma,t_0)=(C_0,0.08,30)}$ for $0\leq C_0\leq 0.004$. The blue curve on the right depicts $\frac{\partial h}{\partial \gamma}|_{(C_0,\gamma,t_0)=(0.004,0.08,30)}$ for $0\leq t\leq 100$. (D) The red curve on the left depicts $\frac{\partial W}{\partial t_0}|_{(C_0,\gamma,t_0)=(C_0,0.08,30)}$ for $0\leq C_0\leq 0.004$. The blue curve on the right depicts $\frac{\partial h}{\partial t_0}|_{(C_0,\gamma,t_0)=(0.004,0.08,30)}$ for $0\leq t\leq 100$. Note that $W$ is a scalar multiple of the total elastic bending energy of the membrane and $h(t)$ is a scalar multiple of the mean curvature at the corresponding points on the membrane.}
    \label{fig:Figure8}
\end{figure}

\section{Results and conclusions}\label{sec:conclusion}
\subsection{Insights obtained from sensitivity analyses on membrane shapes}\label{sec:s-insights}
\begin{itemize}
    \item \textbf{Remarks on the size of the domain.}\\
    As discussed earlier, we use a MATLAB BVP solver to approximate the solution to our system of ODEs. 
    For this reason, the dimensionless size of the domain (that is, the interval on which we want to solve the system), the mesh size, and the initial guess for solution (which is given to the solver as an input) can play key roles in performance of the method and accuracy of the results. 
    In each of our numerical experiments we normally start with dividing the interval into about 100 subintervals and then use a finer mesh if we run into trouble.
    As expected, we observed that for larger domain size, finer partitions are needed for a good performance. 
    In our simulations we adhered to the following rules with regard to the total size of the domain:
    \begin{enumerate}
        \item Dimensionless Arc-length Formulation: Total size is taken to be 2-5 times the size of the coated region.
        \item Dimensionless Area Formulation: Total size is taken to be 4-25 times the size of the coated region.
    \end{enumerate}
    It is important to mention that there are experimental measurements that can provide an estimate of the size of the coat \cite{nickels2017vivo}. 
    In some of our simulations, we used $a_0=20106$ $nm^2$ as an estimate of the area of the coat \cite{hassinger2017}.
    The corresponding dimensionless area is obtained by first choosing a number $R_0$ and then computing $\frac{a_0}{2\pi R_0^2}$. 
    We mention in passing that, although theoretically the value chosen for $R_0$ should not have any effect on the solution of the system, in practice, when using a numerical method to solve the problem, the value of $R_0$ can affect the performance of the numerical solver. 
    Our chosen values for $R_0$ can be found in \Cref{tab:Table 1,tab:Table 2,tab:Table 3,tab:Table 4} displayed in \S \ref{sec:numer-results} and Appendix \ref{app:numericalarea}.\\\\
     We know that the area of a spherical cap corresponding to arc-length $s$ is approximately equal to $\pi s^2$ (assuming $s<<$radius of the sphere). For this reason, it is not completely unreasonable to use the equation $\pi s_0^2=20106$ to come up with an estimate of the length corresponding to the coated region (which gives $s_0\approx 80$ $nm$).

    \item \textbf{Should we expect the area formulation and the arc-length formulation produce the exact same results?}\\
    Let's assume we have fixed the values of $a_0$, $s_0$, $C_0$, and $\xi$. An important question is the following: Is it reasonable to expect that we should obtain the exact same answers using arc-length formulation and area formulation?  
    A little bit of deliberation tells us that the answer is no! For one thing, what the exact relationship between $a_0$ and $s_0$ should be is not known in advance. More importantly, although the spontaneous curvature function used for the area formulation has the same form as the spontaneous curvature function used in arc-length formulation, the area spontaneous curvature used in our simulations is not exactly the function obtained by transforming $C(s)=0.5C_0F(s)$ into a function of $a$ using the relationship $a=g(s)=2\pi\int_0^sr(u)\,du$. As discussed in Remark \ref{remmay242}, the area spontaneous curvature function that would have produced the same answer would have been $\hat{C}(a)=0.5C_0F(g^{-1}(a))$. For example, in the case where $F(s)=[1-\tanh[\xi(s-s_0)]]$, we would obtain $\hat{C}(a)=0.5C_0[1-\tanh[\xi(g^{-1}(a)-s_0)]]$ which, in part due to the presence of the function $g^{-1}$, does not necessarily have the same form as the form that is actually prescribed for $C(a)$ in our simulations.
    
    \item \textbf{On the importance of boundary conditions and the value of $\boldsymbol{\tilde{\lambda}}_0$.}\\
    As it was noted in previous sections, we ran our simulations with two different sets of boundary conditions representing distinct physical constraints/assumptions. 
    Unfortunately, as opposed to the case of initial value problems, there is no general mathematical theory that can be used to ensure the existence and uniqueness of solutions to boundary value problems such as the one studied in this work. 
    One thing that became clear to us was that the performance of the MATLAB BVP solver `bvp4c' was highly sensitive to the chosen boundary conditions, in particular to the value of $\tilde{\lambda}_0$. 
    In fact, in some cases, a small change in the value of $\tilde{\lambda}_0$ can result in the appearance of linear algebraic systems with singular coefficient matrices in the process of numerically solving the ODE system. 
    At this point, it is not clear whether this is a theoretical issue related to existence of solutions, or whether this should be explained by exploring the stability properties of the numerical method used to solve the system.
    
    \item \textbf{A cleverly chosen spontaneous curvature can smear out solution sensitivities.}\\
    As we mentioned before, whether or not the spontaneous curvature function is mollifying can affect the solution sensitivities. 
    Our simulations provide numerical evidence for the conjecture that the smoother the transition between nonzero part of $F(u)$ to the zero part of $F(u)$ is, the less likely it will be to have abrupt jumps in the graphs of solution sensitivities. 
    Notice that the shape of the function $F$ over the transition region can be tuned by cleverly choosing the parameters $\gamma$ and $t_0$ (or $\alpha_0$). 
    
    \item \textbf{Is there any correlation between sensitivity diagrams and the final shape of the membrane?}\\
    One motivation of this work was to gain more insight into the circumstances that would result in bud-shaped membranes versus those that would give pearl-shaped membranes. 
    To that end, we performed various numerical experiments to see whether we could find evidence indicating positive or negative correlation between input data of the problem (such as type of the spontaneous curvature, parameters used in the expression of the spontaneous curvature, and type of the boundary conditions) and the final shape of the membrane. \\\\
    The following observations/conjectures are in agreement with all of our results, parts of which are depicted in \Cref{fig:Figure10}.
    \begin{enumerate}
        
        \item Our results provide numerical evidence for the conjecture that there might be a connection between the behavior of the energy sensitivity with respect to $C_0$ and the appearance of turning points in the generating curve of the surface of the membrane and formation of pearls (loops).  In particular, as it becomes evident by comparing the diagrams displayed in panels A, B, C, D, and G of \Cref{fig:Figure10} with those displayed in panels E, F, and H,  we observed that when $W$ did not have any critical point (as a function of $C_0$), that is, when the energy sensitivity function with respect to $C_0$ did not intersect the horizontal axis, no loops were formed.

        \item Previously we observed that using a mollifying spontaneous curvature function can smear out the solution sensitivities. However, the type of spontaneous curvature function (that is type I or type II) or whether the function is mollifying or not, does not appear to preclude the possibility of formation of  pearls. 
        
        \item For a fixed domain size and coat size, it is more likely to see an energy sensitivity function with oscillatory behavior about the horizontal axis (indicating existence of zeros) when we use the second type of boundary conditions. Indeed, in our experiments, every time we used the first set of boundary conditions, we noticed that the resulting energy sensitivity function (as a function of $C_0$) will not intersect the horizontal axis, and subsequently no pearls were formed. 
        \end{enumerate}

\end{itemize}

\begin{sidewaysfigure}
    \centering
    \includegraphics[scale=0.45]{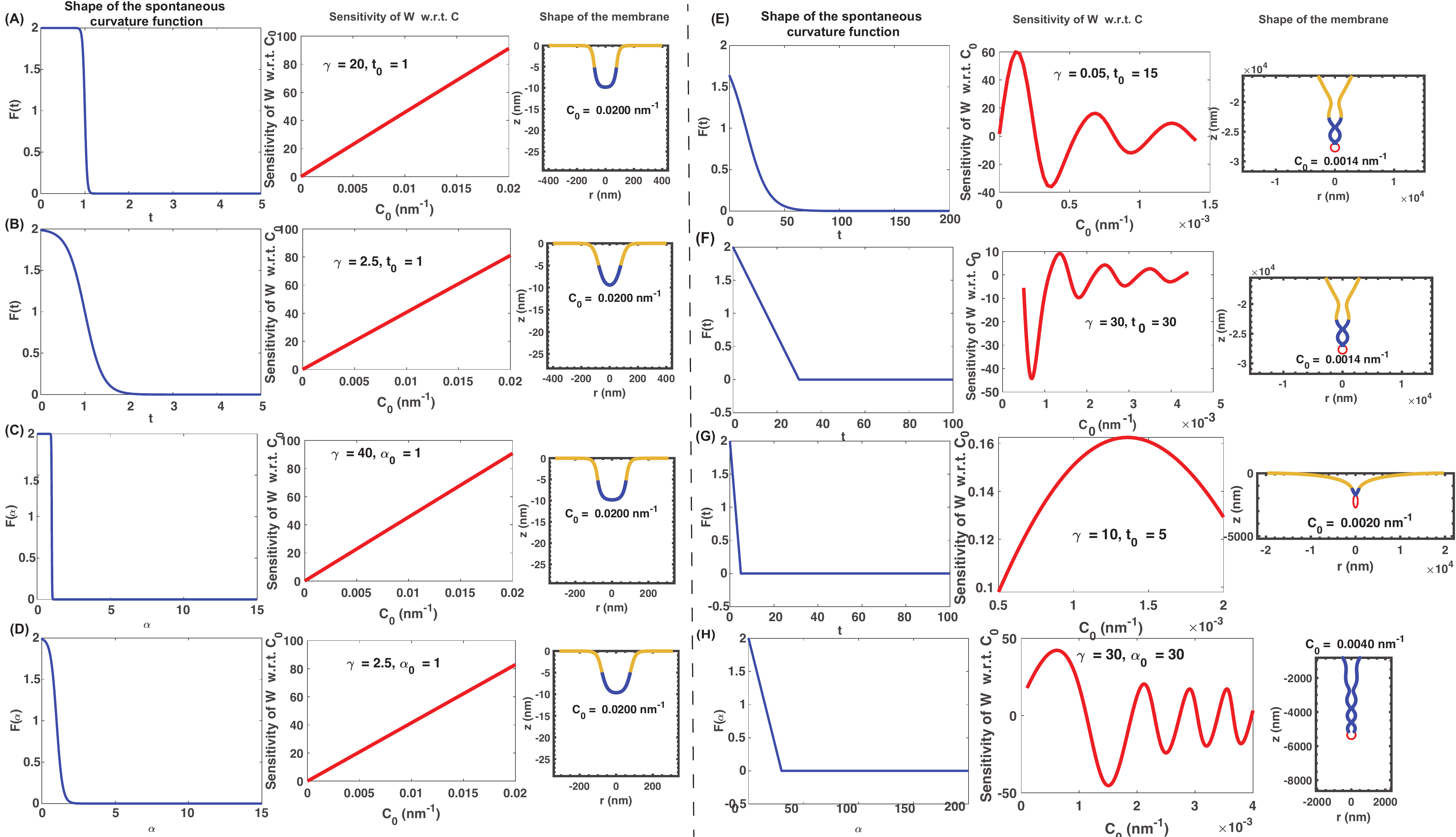}
    \centering
    \caption{(A) Arc-length Formulation, Type I Spontaneous Curvature, Type I B.C. 
    (B) Arc-length Formulation, Mollifying Type I Spontaneous Curvature, Type I B.C. 
    (C) Area Formulation, Type I Spontaneous Curvature, Type I B.C.
    (D) Area Formulation, Mollifying Type I Spontaneous Curvature, Type I B.C.
    (E) Arc-length Formulation, Type I Spontaneous Curvature, Type II B.C.
    (F) Arc-length Formulation, Type II Spontaneous Curvature, Type II B.C.
    (G) Arc-length Formulation, Type II Spontaneous Curvature, Type II B.C.
    (H) Area Formulation, Type II Spontaneous Curvature, Type II B.C.}
    \label{fig:Figure10}
\end{sidewaysfigure}

\subsection{Examining a conjecture related to the `pearling' transition} \label{sec:disprove}

In our discussions with researchers working on the subject, we noticed that some implicitly believe in the conjecture that ``the slope of the transition part of the spontaneous curvature function is a key factor in whether or not pearls will be formed; in particular, there is a correlation between formation of pearls and not having sharp transitions (big slopes) in the spontaneous curvature function used in the numerical solution of the boundary value problem". \\\\
The purpose of this section is to provide a simple argument that disproves the validity of the above conjecture in the generality stated above. Indeed, in what follows we will show that there exist spontaneous curvature functions with very large slopes on their transition regions that ultimately result in formation of pearls.
To be concrete, we focus on the dimensionless arc-length parametrization, however, an analogous argument can be applied to the dimensionless area parametrization.\\\\
Let $(C_0, R_0,\gamma, t_0)$ be a set of parameters with $\gamma t_0>3$ that results in formation of pearls on the domain $[0,T]$. Note that since $\gamma t_0>3$, the value of the spontaneous curvature function at $0$ is approximately $C_0R_0$ and it decreases to near zero over the interval $[0, t_0+\frac{3}{\gamma}]$. As we discussed  in \S \ref{sec:axsi-param}, if $x(t)$ and $y(t)$ are the first two components of the corresponding solution (note that $R_0x(t)$ and $R_0y(t)$ are the coordinates of the generating curve of the surface of the membrane), then for any constant $\eta>0$, the functions $\frac{1}{\eta}x(\eta t)$ and $\frac{1}{\eta}y(\eta t)$ will be the first two components of the solution on $[0,\frac{T}{\eta}]$ using parameters  $(C_0,\eta R_0,\eta \gamma, \frac{t_0}{\eta})$. In particular, the shape of the solution curve (and the corresponding surface) using these new parameters will be the same as the shape of the solution curve (and the corresponding surface) using the original parameters. If one has pearls, then the other will also have pearls. Now note that $(\eta \gamma)\frac{t_0}{\eta}=\gamma t_0>3$ and so the value of the corresponding spontaneous curvature at $0$ is approximately  $(C_0)(\eta R_0)$ and the spontaneous curvature function decreases to nearly zero over the interval $[0, \frac{t_0}{\eta}+\frac{3}{\eta \gamma}]$. We can set the number $\eta$ to be as large as we want. The larger the $\eta$, the smaller the interval of transition will be and the larger the starting value $\eta C_0R_0$ will be. That is, for a very large $\eta$, the corresponding spontaneous curvature function must decrease from the huge value $\eta C_0R_0$ to zero over the very small interval $[0, \frac{t_0}{\eta}+\frac{3}{\eta \gamma}]$, which means it will have a huge slope. Nevertheless, since the original parameters resulted in the formation of pearls, this new spontaneous curvature function will also result in the formation of pearls.

\subsection{Concluding remarks}\label{sec:s-future-directions}
In this study, we conducted sensitivity analysis on the well-known Helfrich model for lipid bilayer bending in the context of the spontaneous curvature function. 
We observed some interesting phenomena in our numerical experiments that relate the sensitivity of the energy with respect to the free parameters in the spontaneous curvature function to the shape of the membrane.
Given the wide usage of the Helfrich model for simulating membrane bending phenomena that have been reported experimentally, our approach of using sensitivity analysis can provide some insight into how one can design input functions and parameters to this system of ODEs.
Our experiments also led us to make certain new, arguably nonobvious, conjectures about the behavior of the solution.
Clearly even if results of numerical experiments using a million different sets of inputs display positive correlation between certain output variables, that does not necessarily mean that theoretically there should be positive correlation between the output variables regardless of what the inputs are. 
Nevertheless, such results may help us make conjectures that were not immediately obvious from the outset, and this is in fact a major way that science progresses.  
This work by no means should be viewed as the final word on the parametric sensitivity analysis of the shape equations in Helfrich energy model. 
Rather, we suggest that it is merely a first step toward better understanding the role of parameters involved in shape equations, particularly in the context of numerical simulations. 
Indeed, future directions include identifying suitable `bump' functions for the spontaneous curvature function and extension of these methods to solutions in general coordinates. 

\section*{Acknowledgments}
The authors would like to thank Haleh Alimohamadi, Jennifer Fromm, Christopher Lee, Can Uysalel, Ritvik Vasan, Cuncheng Zhu, and other members of the Rangamani group for critical discussions and feedback on the manuscript. M.H. was supported in part by NSF DMS/CM Award 1620366 and NSF DMS/MB Award 1934411. This work was funded in part by R01GM132106 from the National Institutes of Health and FA9550-18-1-0051 from the Air Force Office of Scientific Research to P.R.

\newpage
\begin{appendices}
\section{Some numerical results for dimensionless area parametrization}\label{app:numericalarea}
In this appendix we will present some of our numerical results obtained using the dimensionless area parametrization.
\subsection{Dimensionless area formulation -- Type I spontaneous curvature}\label{sec:s-nondim-area-1}

    The spontaneous curvature function used is:
    \begin{equation}
    c(\alpha)=0.5C_0R_0\big[1-\tanh\big[\gamma(\alpha-\alpha_0)\big]\big].
    \end{equation}
     Note that for each choice of $C_0$, $c(\alpha)$ is a constant multiple of $F(\alpha)=1-\tanh\big[\gamma(\alpha-\alpha_0)\big]$. So $F(\alpha)$ is the function that determines the shape of the spontaneous curvature. The graph of $F(\alpha)$ is depicted in  \Cref{fig:Figure4}.\\\\
    Boundary conditions used for the original unknowns are:
\begin{align}
\begin{cases}
&x_1(0^+)=0,\quad x_3(0^+)=0,\quad x_5(0^+)=0,\\
&x_2(\alpha_{max})=0,\quad x_3(\alpha_{max})=0,\quad x_6(\alpha_{max})=\tilde{\lambda}_0.
\end{cases}
\end{align}
  The boundary conditions for the corresponding sensitivities are all set to be zero.\\
  Our choices for the input parameters are given in \Cref{tab:Table 2}.
    \begin{table}[H]
    \centering
    \begin{tabular}{| >{\centering\arraybackslash}m{2in} | >{\centering\arraybackslash}m{2in} |@{}m{0cm}@{}}
\hline Parameter & Input used\\
\hline $R_0$ $(nm)$ & $\frac{400}{\sqrt{50}}$&\\ 
\hline $C_0$ $(nm^{-1})$ & various values in $[0, 0.02]$&\\
\hline $\gamma=\xi*2\pi R_0^2$ & $40$&\\
\hline $\alpha_0=\frac{a_0}{2\pi R_0^2}$ & $1$&\\
\hline $\tilde{\lambda}_0=\lambda_0R_0^2/\kappa_0$ & $6.4$&\\
\hline Domain $=[0,\alpha_{max}]$ & $[0,15]$&\\
\hline
\end{tabular}
\caption{Parameters used in the model. The corresponding diagrams are depicted in \Cref{fig:Figure4}. }
\label{tab:Table 2}
\end{table}
\noindent $R_0$ is chosen such that $\alpha_0=1$. A brief explanation of the choice $a_0=2\pi R_0^2\alpha_0\approx 20106$ and $\alpha_{max}=15$ is given in \S \ref{sec:conclusion}. $\gamma$ is chosen such that there will be a sharp transition at $\alpha_0=1$.\\

We organize our results as follows: graphs of energy sensitivities with respect to parameters $C_0$, $\lambda$, and $t_0$ are displayed on the left column of \Cref{fig:Figure4}B, C, and D, respectively; the graphs of the dimensionless curvature sensitivities with respect to parameters involved are depicted on the right column of these panels. Again we make two key observations. First, notice that for our choice of the type of boundary conditions (B.C. of type I, \Cref{eqn-area-bc-1}) and spontaneous curvature function, the energy sensitivity graphs do not intersect the horizontal axis. Second, notice that for our choice of boundary conditions and spontaneous curvature function (which possesses a sharp transition), there is an abrupt change in curvature sensitivities near the location where the sharp transition in spontaneous curvature function occurs. The graphs of curvature sensitivities will smear out if we smooth the transition in the  spontaneous curvature function. 


\begin{figure}[H]
    \centering
    \includegraphics[height=5in,width=0.7\textwidth]{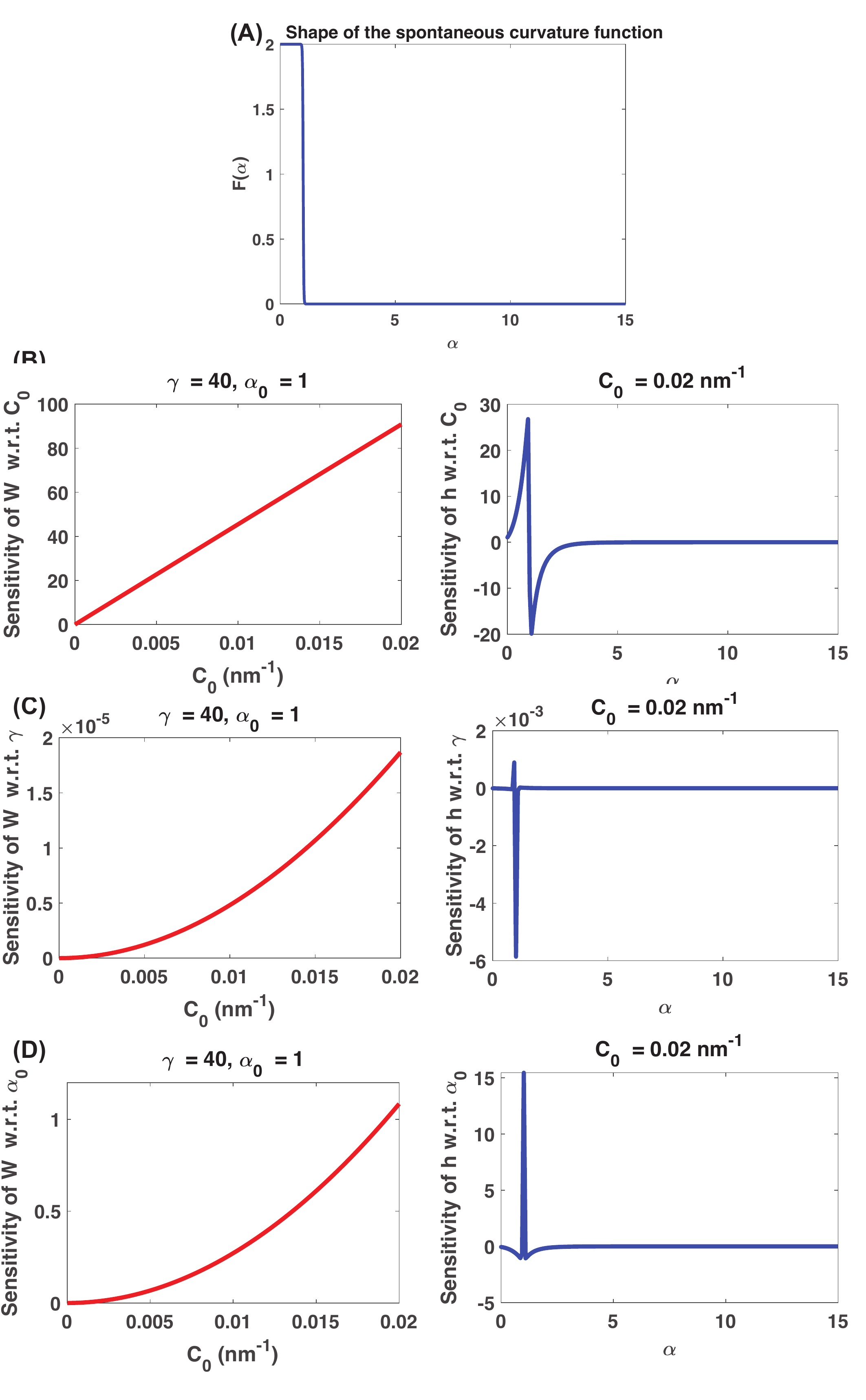}
    \caption{(A) Shape of the spontaneous curvature function with $\gamma=40$ and $\alpha_0=1$.  (B) The red curve on the left depicts $\frac{\partial W}{\partial C_0}|_{(C_0,\gamma,\alpha_0)=(C_0,40,1)}$ for $0\leq C_0\leq 0.02$. The blue curve on the right depicts $\frac{\partial h}{\partial C_0}|_{(C_0,\gamma,\alpha_0)=(0.02,40,1)}$ for $0\leq \alpha\leq 15$. (C) The red curve on the left depicts $\frac{\partial W}{\partial \gamma}|_{(C_0,\gamma,\alpha_0)=(C_0,40,1)}$ for $0\leq C_0\leq 0.02$. The blue curve on the right depicts $\frac{\partial h}{\partial \gamma}|_{(C_0,\gamma,\alpha_0)=(0.02,40,1)}$ for $0\leq \alpha\leq 15$. (D) The red curve on the left depicts $\frac{\partial W}{\partial t_0}|_{(C_0,\gamma,\alpha_0)=(C_0,40,1)}$ for $0\leq C_0\leq 0.02$. The blue curve on the right depicts $\frac{\partial h}{\partial \alpha_0}|_{(C_0,\gamma,t_0)=(0.02,40,1)}$ for $0\leq \alpha\leq 15$. Note that $W$ is a scalar multiple of the total elastic bending energy of the membrane and $h(\alpha)$ is a scalar multiple of the mean curvature at the corresponding points on the membrane.}
    \label{fig:Figure4}
\end{figure}

\subsection{Dimensionless area formulation - Mollifying type I spontaneous curvature}\label{sec:s-nondim-area-molly1}

All parameters are exactly the same as in \S \ref{sec:s-nondim-area-1} except $\gamma$. This time we choose $\gamma$ to be equal to $2.5$ such that $\alpha_0\gamma<3$. This choice of $\gamma$ and $\alpha_0$ smooths the transition in the spontaneous curvature function (see \Cref{fig:Figure5}A). As a result we will see that the graphs of curvature sensitivities (diagrams in the second column of \Cref{fig:Figure5}B,C, and D) are smeared out (compared with what was observed in \S \ref{sec:s-nondim-area-1}). Notice that this particular change in the shape of the spontaneous curvature function did not have any effect on the number of horizontal intercepts of the energy sensitivities.

\begin{figure}[H]
    \centering
    \includegraphics[height=5in,width=0.7\textwidth]{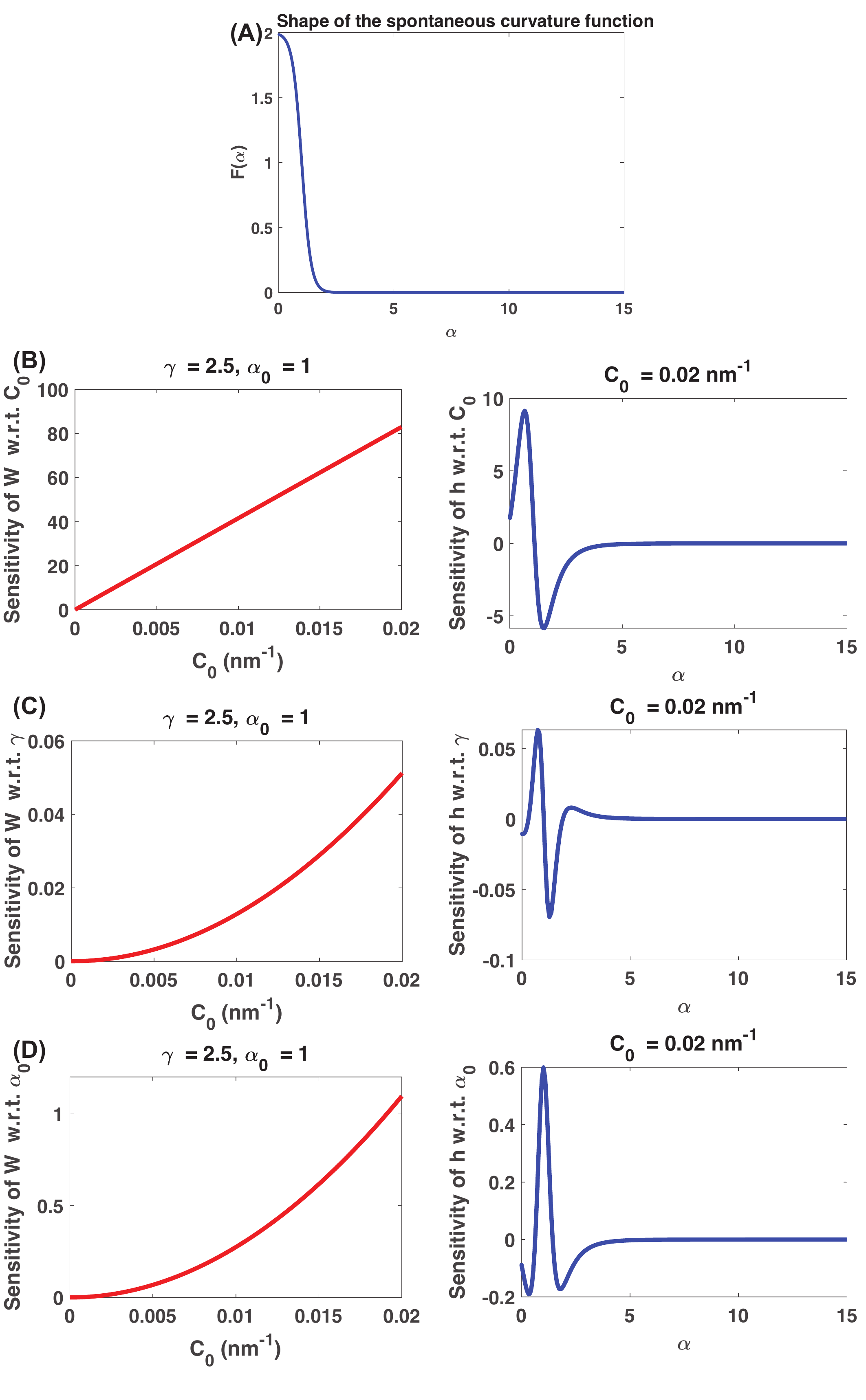}
    \caption{(A) Shape of the spontaneous curvature function with $\gamma=2.5$ and $\alpha_0=1$.  (B) The red curve on the left depicts $\frac{\partial W}{\partial C_0}|_{(C_0,\gamma,\alpha_0)=(C_0,2.5,1)}$ for $0\leq C_0\leq 0.02$. The blue curve on the right depicts $\frac{\partial h}{\partial C_0}|_{(C_0,\gamma,\alpha_0)=(0.02,2.5,1)}$ for $0\leq \alpha\leq 15$. (C) The red curve on the left depicts $\frac{\partial W}{\partial \gamma}|_{(C_0,\gamma,\alpha_0)=(C_0,2.5,1)}$ for $0\leq C_0\leq 0.02$. The blue curve on the right depicts $\frac{\partial h}{\partial \gamma}|_{(C_0,\gamma,\alpha_0)=(0.02,2.5,1)}$ for $0\leq \alpha\leq 15$. (D) The red curve on the left depicts $\frac{\partial W}{\partial \alpha_0}|_{(C_0,\gamma,\alpha_0)=(C_0,2.5,1)}$ for $0\leq C_0\leq 0.02$. The blue curve on the right depicts $\frac{\partial h}{\partial \alpha_0}|_{(C_0,\gamma,\alpha_0)=(0.02,2.5,1)}$ for $0\leq \alpha\leq 15$. Note that $W$ is a scalar multiple of the total elastic bending energy of the membrane and $h(\alpha)$ is a scalar multiple of the mean curvature at the corresponding points on the membrane.}
    \label{fig:Figure5}
\end{figure}

\subsection{Dimensionless area formulation -- Type II spontaneous curvature}\label{sec:s-nondim-area-2}
The spontaneous curvature function used is:
    \begin{equation}
    c(\alpha)=-0.5C_0R_0\frac{\alpha-\alpha_0}{\alpha_0}\big[1-\tanh\big[\gamma(\alpha-\alpha_0)\big]\big].
    \end{equation}
    Note that for each choice of $C_0$, $c(t)$ is a constant multiple of $F(\alpha)=-\frac{\alpha-\alpha_0}{\alpha_0}\big[1-\tanh\big[\gamma(\alpha-\alpha_0)\big]\big]$. So $F(\alpha)$ is the function that determines the shape of the spontaneous curvature. The graph of $F(\alpha)$ is depicted in  \Cref{fig:Figure7}A.\\\\
    Boundary conditions used for the original unknowns are:
\begin{align}
\begin{cases}
&x_1(0^+)=\sin{(0.3\pi)},\quad  x_3(0^+)=0.3\pi,\\
&x_2(\alpha_{max})=0,\quad x_3(\alpha_{max})=0,\quad x_5((\alpha_{max})=0,\quad x_6((\alpha_{max})=0.
\end{cases}
\end{align}
  The boundary conditions for the corresponding sensitivities are all set to be zero.\\
  Our choices for the input parameters are given in \Cref{tab:Table 4}.
    \begin{table}[H]
    \centering
    \begin{tabular}{| >{\centering\arraybackslash}m{2in} | >{\centering\arraybackslash}m{2in} |@{}m{0cm}@{}}
\hline Parameter & Input used\\
\hline $R_0$ $(nm)$& $200$&\\ 
\hline $C_0$ $(nm^{-1})$ &various values in $[0,\frac{4}{1000}]$ &\\
\hline $\gamma=\xi (2\pi R_0^2)$ & $30$&\\
\hline $\alpha_0=\frac{a_0}{2\pi R_0^2}$ & $30$&\\
\hline Domain $=[0,\alpha_{max}]$ & $[0,200]$&\\
\hline
\end{tabular}
\caption{Parameters used in the model. The corresponding diagrams are depicted in \Cref{fig:Figure7}.}
\label{tab:Table 4}
\end{table}

\noindent This particular choice of parameters is in agreement with \cite{yuan2020}.

We organize our results as follows: graphs of energy sensitivities with respect to parameters $C_0$, $\lambda$, and $t_0$ are displayed on the left column of \Cref{fig:Figure8}B, C, and D, respectively; the graphs of the dimensionless curvature sensitivities with respect to parameters involved are depicted on the right column of these panels. We make two observations. First, notice that for our choice of the type of boundary conditions (B.C. of type II, \Cref{eqn-area-bc-2}) and spontaneous curvature function, the energy sensitivity graph with respect to $C_0$ indeed intersects the horizontal axis. Second, notice that for our choice of boundary conditions and spontaneous curvature function (which possesses a sharp transition), the graphs of curvature sensitivities have sharp bends particularly near the location where the sharp transition in spontaneous curvature function occurs. The graphs of curvature sensitivities will be smoother near $\alpha_0$ if we smooth the sharp transition in the spontaneous curvature function.

\begin{figure}[H]
    \centering
    \includegraphics[height=5in,width=0.7\textwidth]{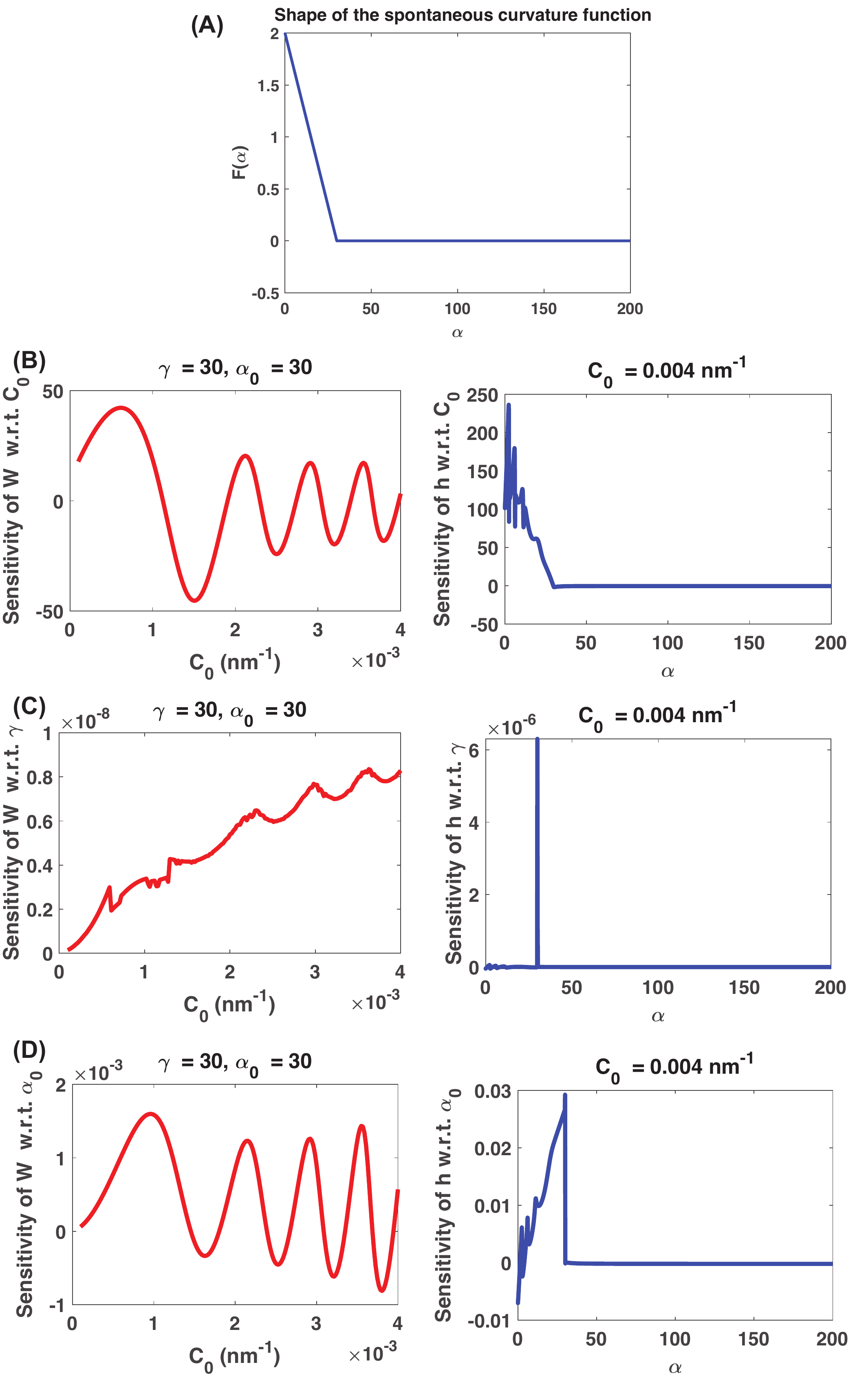}
    \caption{(A) Shape of the spontaneous curvature function with $\gamma=30$ and $\alpha_0=30$.  (B) The red curve on the left depicts $\frac{\partial W}{\partial C_0}|_{(C_0,\gamma,\alpha_0)=(C_0,30,30)}$ for $0\leq C_0\leq 0.004$. The blue curve on the right depicts $\frac{\partial h}{\partial C_0}|_{(C_0,\gamma,\alpha_0)=(0.004,30,30)}$ for $0\leq \alpha\leq 200$. (C) The red curve on the left depicts $\frac{\partial W}{\partial \gamma}|_{(C_0,\gamma,\alpha_0)=(C_0,30,30)}$ for $0\leq C_0\leq 0.004$. The blue curve on the right depicts $\frac{\partial h}{\partial \gamma}|_{(C_0,\gamma,\alpha_0)=(0.004,30,30)}$ for $0\leq \alpha\leq 200$. (D) The red curve on the left depicts $\frac{\partial W}{\partial t_0}|_{(C_0,\gamma,\alpha_0)=(C_0,30,30)}$ for $0\leq C_0\leq 0.004$. The blue curve on the right depicts $\frac{\partial h}{\partial \alpha_0}|_{(C_0,\gamma,\alpha_0)=(0.004,30,30)}$ for $0\leq \alpha\leq 200$. Note that $W$ is a scalar multiple of the total elastic bending energy of the membrane and $h(\alpha)$ is a scalar multiple of the mean curvature at the corresponding points on the membrane.}
    \label{fig:Figure7}
\end{figure}


\subsection{Dimensionless area formulation -- Mollifying type II spontaneous curvature}\label{sec:s-nondim-area-molly2}
All parameters are exactly the same as \S \ref{sec:s-nondim-area-2} except $\gamma$. This time we choose $\gamma$ to be equal to $0.08$ such that $\alpha_0\gamma<3$. This choice of $\gamma$ and $\alpha_0$ further smooths the transition in the spontaneous curvature function (see \Cref{fig:Figure9}A). As a result we will see that the graphs of curvature sensitivities (diagrams in the second column of \Cref{fig:Figure9}B, C, and D) are smoother near $\alpha_0=30$ (compared with what was observed in \S \ref{sec:s-nondim-area-2}). 
\begin{figure}[H]
    \centering
    \includegraphics[height=5in,width=0.7\textwidth]{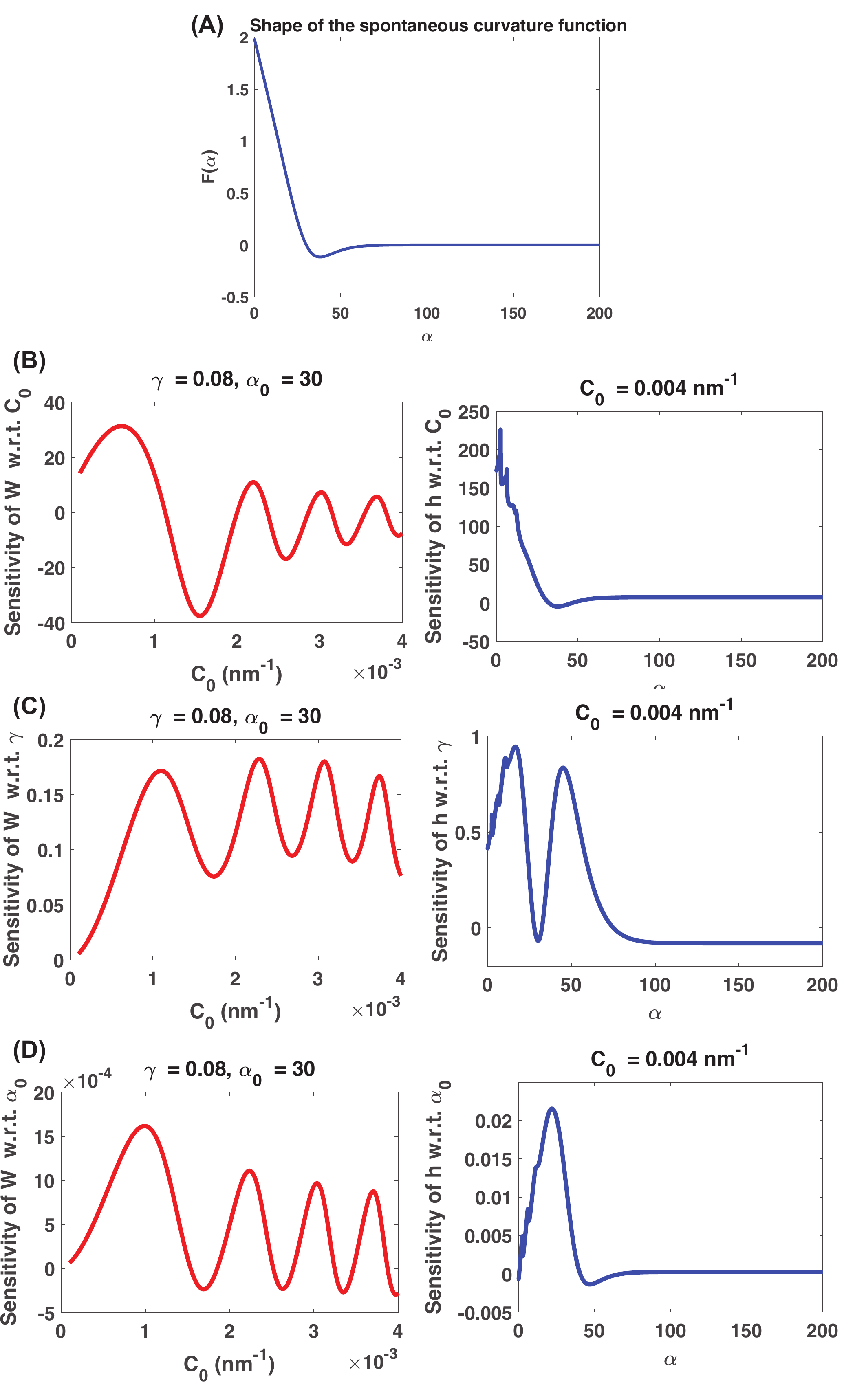}
    \caption{(A) Shape of the spontaneous curvature function with $\gamma=0.08$ and $\alpha_0=30$.  (B) The red curve on the left depicts $\frac{\partial W}{\partial C_0}|_{(C_0,\gamma,\alpha_0)=(C_0,0.08,30)}$ for $0\leq C_0\leq 0.004$. The blue curve on the right depicts $\frac{\partial h}{\partial C_0}|_{(C_0,\gamma,\alpha_0)=(0.004,0.08,30)}$ for $0\leq \alpha\leq 200$. (C) The red curve on the left depicts $\frac{\partial W}{\partial \gamma}|_{(C_0,\gamma,\alpha_0)=(C_0,0.08,30)}$ for $0\leq C_0\leq 0.004$. The blue curve on the right depicts $\frac{\partial h}{\partial \gamma}|_{(C_0,\gamma,\alpha_0)=(0.004,0.08,30)}$ for $0\leq \alpha\leq 200$. (D) The red curve on the left depicts $\frac{\partial W}{\partial t_0}|_{(C_0,\gamma,\alpha_0)=(C_0,0.08,30)}$ for $0\leq C_0\leq 0.004$. The blue curve on the right depicts $\frac{\partial h}{\partial \alpha_0}|_{(C_0,\gamma,\alpha_0)=(0.004,0.08,30)}$ for $0\leq \alpha\leq 200$. Note that $W$ is a scalar multiple of the total elastic bending energy of the membrane and $h(\alpha)$ is a scalar multiple of the mean curvature at the corresponding points on the membrane.}
    \label{fig:Figure9}
\end{figure}

\section{The derivatives} \label{app:derivatives}
For the sake of completeness and also facilitating the verification of our results, here we list the expressions for the various derivatives needed in our analysis.
\begin{itemize}
    \item Arc-length Formulation
    \begin{align}
D_{\vec{x}}\textbf{F}&=\begin{bmatrix}\frac{\partial \textbf{F}}{\partial x_1}& \frac{\partial \textbf{F}}{\partial x_2}& \cdots&\frac{\partial \textbf{F}}{\partial x_6}\end{bmatrix}\notag\\
&=\begin{bmatrix}
0 & 0 & \sin{x_3}& 0 & 0 & 0\\
0 & 0 & -\cos{x_3}& 0 & 0 & 0\\
\frac{-\sin{x_3}}{x_1^2}& 0 & \frac{\cos{x_3}}{x_1}& -2 & 0 & 0\\
\frac{x_5}{x_1^2}& 0 & 0& 0 & \frac{-1}{x_1} & 0\\
\frac{\partial f_5}{\partial x_1}& 0 & 4x_1(x_4-c)(\frac{-\cos{x_3}}{x_1})(x_4-\frac{\sin{x_3}}{x_1}) & \frac{\partial f_5}{\partial x_4}& 0 & \frac{-2x_1x_4}{\tilde{\kappa}}\\
0 & 0 & 0 & -2\tilde{\kappa}\dot{c}& 0 & 0
\end{bmatrix},
\end{align}
where
\begin{align}
&\frac{\partial f_5}{\partial x_1}=-2x_4\big[(x_4-c)^2+\frac{x_6}{\tilde{\kappa}})\big]+2(x_4-c)x_4^2+2(x_4-c)(x_4-\frac{\sin{x_3}}{x_1})^2\\
&\hspace{6cm}+4x_1(x_4-c)(\frac{\sin{x_3}}{x_1^2})(x_4-\frac{\sin{x_3}}{x_1}),\notag\\
&\frac{\partial f_5}{\partial x_4}=-2x_1(x_4-c)^2-\frac{2x_1x_6}{\tilde{\kappa}}+2x_1x_4^2
+2x_1(x_4-\frac{\sin{x_3}}{x_1})^2\\
&\hspace{6cm}+4x_1(x_4-c)[x_4-\frac{\sin{x_3}}{x_1}].\notag
\end{align}
\begin{align}
D_{\dot{\vec{x}}}\textbf{F}&=\begin{bmatrix}\frac{\partial \textbf{F}}{\partial \dot{x}_1}& \frac{\partial \textbf{F}}{\partial \dot{x}_2}& \cdots&\frac{\partial \textbf{F}}{\partial \dot{x}_6}\end{bmatrix}=\textrm{Id}\,\quad \textrm{(Id $=$ identity matrix)}.
\end{align}
\begin{align}
    \frac{\partial \textbf{F}}{\partial p_j}=
    \begin{bmatrix}
    0\\
    0\\
    0\\
    -\frac{\partial \dot{c}}{\partial p_j}\\
    4x_1x_4(x_4-c)\frac{\partial c}{\partial p_j}-
    2x_1[x_4^2+(x_4-\frac{\sin{x_3}}{x_1})^2]\frac{\partial c}{\partial p_j}\\
    -2\tilde{\kappa}x_4\frac{\partial \dot{c}}{\partial p_j}+
    2\tilde{\kappa}[\frac{\partial c}{\partial p_j}\dot{c}+c\frac{\partial \dot{c}}{\partial p_j}]
    \end{bmatrix}.
    \end{align}
    \begin{align}
    &w(\vec{x}(\vec{p}),\vec{p})=[x_4(t)-c(t)]^2x_1(t)\\
    &\hspace{2cm}\longrightarrow \frac{D w}{D p_j}=2[x_4-c](\frac{\partial x_4}{\partial p_j}-\frac{\partial c}{\partial p_j})x_1+(x_4-c)^2\frac{\partial x_1}{\partial p_j}.\notag
\end{align}
    \item Area Formulation:
    \begin{align}
&D_{\vec{x}}\textbf{F}=\begin{bmatrix}\frac{\partial \textbf{F}}{\partial x_1}& \frac{\partial \textbf{F}}{\partial x_2}& \cdots&\frac{\partial \textbf{F}}{\partial x_6}\end{bmatrix}\notag\\
&=\begin{bmatrix}
\frac{cos{x_3}}{x_1^2} & 0 & \frac{\sin{x_3}}{x_1}& 0 & 0 & 0\\
\frac{sin{x_3}}{x_1^2} & 0 & \frac{-\cos{x_3}}{x_1}& 0 & 0 & 0\\
\frac{2x_4}{x_1^2}-\frac{2\sin{x_3}}{x_1^3}& 0 & \frac{\cos{x_3}}{x_1^2}& \frac{-2}{x_1} & 0 & 0\\
\frac{2x_5}{x_1^3}& 0 & 0& 0 & \frac{-1}{x_1^2} & 0\\
4(x_4-c)(\frac{\sin{x_3}}{x_1^2})(x_4-\frac{\sin{x_3}}{x_1})& 0 & 4(x_4-c)(\frac{-\cos{x_3}}{x_1})(x_4-\frac{\sin{x_3}}{x_1}) & \frac{\partial f_5}{\partial x_4}& 0 & \frac{-2x_4}{\tilde{\kappa}}\\
0 & 0 & 0 & -2\tilde{\kappa}\dot{c}& 0 & 0
\end{bmatrix},
\end{align}
where
\begin{align}
\frac{\partial f_5}{\partial x_4}=-2(x_4-c)^2-\frac{2x_6}{\tilde{\kappa}}+2x_4^2+2(x_4-\frac{\sin{x_3}}{x_1})^2+4(x_4-c)[x_4-\frac{\sin{x_3}}{x_1}].
\end{align}
\begin{align}
D_{\dot{\vec{x}}}\textbf{F}&=\begin{bmatrix}\frac{\partial \textbf{F}}{\partial \dot{x}_1}& \frac{\partial \textbf{F}}{\partial \dot{x}_2}& \cdots&\frac{\partial \textbf{F}}{\partial \dot{x}_6}\end{bmatrix}=\textrm{Id}.\,
\end{align}
\begin{align}
    \frac{\partial \textbf{F}}{\partial p_j}=
    \begin{bmatrix}
    0\\
    0\\
    0\\
    -\frac{\partial \dot{c}}{\partial p_j}\\
    4x_4(x_4-c)\frac{\partial c}{\partial p_j}-
    2[x_4^2+(x_4-\frac{\sin{x_3}}{x_1})^2]\frac{\partial c}{\partial p_j}\\
    -2\tilde{\kappa}x_4\frac{\partial \dot{c}}{\partial p_j}+
    2\tilde{\kappa}[\frac{\partial c}{\partial p_j}\dot{c}+c\frac{\partial \dot{c}}{\partial p_j}]
    \end{bmatrix}.
    \end{align}
    \begin{align}
    w(\vec{x}(\vec{p}),\vec{p})=[x_4(\alpha)-c(\alpha)]^2\longrightarrow \frac{D w}{D p_j}=2[x_4-c](\frac{\partial x_4}{\partial p_j}-\frac{\partial c}{\partial p_j}).
\end{align}
   
    \item Type I Spontaneous Curvature

\begin{align}
&c(u)=\frac{R_0C_0}{2}\big[1-\tanh{[\gamma(u-u_0)]}\big].\\
&\dot{c}(u)=-\frac{R_0C_0\gamma}{2}\textrm{sech}^2[\gamma(u-u_0)].\\
& \frac{\partial c}{\partial p_1}=\frac{\partial c}{\partial C_0}=\frac{R_0}{2}\big[1-\tanh[\gamma(u-u_0)]\big].\\
& \frac{\partial c}{\partial p_2}=\frac{\partial c}{\partial \gamma}=\frac{-R_0C_0}{2}(u-u_0)\textrm{sech}^2[\gamma(u-u_0)].\\
& \frac{\partial c}{\partial p_3}=\frac{\partial c}{\partial u_0}=\frac{R_0C_0\gamma}{2}\textrm{sech}^2[\gamma(u-u_0)].
\end{align}

\begin{align}
& \frac{\partial \dot{c}}{\partial p_1}=\frac{\partial \dot{c}}{\partial C_0}=\frac{-R_0\gamma}{2}\textrm{sech}^2[\gamma(u-u_0)].\\
& \frac{\partial \dot{c}}{\partial p_2}=\frac{\partial \dot{c}}{\partial \gamma}\\
&\hspace{0.6cm}=\frac{-R_0C_0}{2}\textrm{sech}^2[\gamma(u-u_0)] +R_0C_0\gamma (u-u_0)\textrm{sech}^2[\gamma(u-u_0)]\tanh{[\gamma(u-u_0)]}.\notag\\
& \frac{\partial \dot{c}}{\partial p_3}=\frac{\partial \dot{c}}{\partial u_0}=-R_0c_0\gamma^2\textrm{sech}^2[\gamma(u-u_0)]\tanh{[\gamma(u-u_0)]}.
\end{align}
 \item Type II Spontaneous Curvature
    
    \begin{align}
&c(u)=\frac{-R_0C_0}{2u_0}(u-u_0)\big[1-\tanh{[\gamma(u-u_0)]}\big].\\
&\dot{c}(u)=\frac{-R_0C_0}{2u_0}(u-u_0)\big[1-\tanh{[\gamma(u-u_0)]}\big]+\frac{R_0C_0\gamma}{2u_0}(u-u_0)\textrm{sech}^2[\gamma(u-u_0)].\\
& \frac{\partial c}{\partial p_1}=\frac{\partial c}{\partial C_0}=\frac{-R_0}{2u_0}(u-u_0)\big[1-\tanh{[\gamma(u-u_0)]}\big].\\
& \frac{\partial c}{\partial p_2}=\frac{\partial c}{\partial \gamma}=\frac{R_0C_0}{2u_0}(u-u_0)^2\textrm{sech}^2[\gamma(u-u_0)].\\
& \frac{\partial c}{\partial p_3}=\frac{\partial c}{\partial u_0}=\frac{R_0C_0}{2u_0^2}(u-u_0)\big[1-\tanh{[\gamma(u-u_0)]}\big]\\
&\qquad+\frac{R_0C_0}{2u_0}\big[1-\tanh{[\gamma(u-u_0)]}\big]-
\frac{R_0C_0\gamma}{2u_0}(u-u_0)\textrm{sech}^2[\gamma(u-u_0)].\notag
\end{align}

\begin{align}
& \frac{\partial \dot{c}}{\partial p_1}=\frac{\partial \dot{c}}{\partial C_0}=\frac{-R_0}{2u_0}\big[1-\tanh{[\gamma(u-u_0)]}\big]+\frac{R_0\gamma}{2u_0}(u-u_0)\textrm{sech}^2[\gamma(u-u_0)].\\
& \frac{\partial \dot{c}}{\partial p_2}=\frac{\partial \dot{c}}{\partial \gamma}=\frac{R_0C_0}{u_0}(u-u_0)\textrm{sech}^2[\gamma(u-u_0)]\\&\hspace{4cm}-\frac{R_0C_0\gamma}{u_0}(u-u_0)^2\textrm{sech}^2[\gamma(u-u_0)]\tanh{[\gamma(u-u_0)]}.\notag\\
& \frac{\partial \dot{c}}{\partial p_3}=\frac{\partial \dot{c}}{\partial u_0}=\frac{R_0C_0}{2u_0^2}\big[1-\tanh{[\gamma(u-u_0)]}\big]-\frac{R_0C_0\gamma}{u_0}\textrm{sech}^2[\gamma(u-u_0)]\\
&\hspace{2cm}-\frac{R_0C_0\gamma}{2u_0^2}(u-u_0)\textrm{sech}^2[\gamma(u-u_0)]\notag\\
&\hspace{2cm}+\frac{R_0C_0\gamma^2}{u_0}(u-u_0)\textrm{sech}^2[\gamma(u-u_0)]\tanh{[\gamma(u-u_0)]}.\notag
\end{align}

\end{itemize}

\end{appendices}

\newpage

\bibliographystyle{abbrv}
\bibliography{ref}

\end{document}